\DeclareFontFamily{OML}{rsfs}{\skewchar\font'177}
\DeclareFontShape{OML}{rsfs}{m}{n}{ <5> <6> rsfs5 <7> <8> <9>
	rsfs7 <10> <10.95> <12> <14.4> <17.28> <20.74> <24.88> rsfs10 }{}
\DeclareMathAlphabet{\mathfs}{OML}{rsfs}{m}{n}

\documentclass[MNA]{mna} 

\usepackage{caption}
\usepackage{booktabs}
\usepackage{bbm}

\newcommand{\BR}{{\mathbb{R}}}

\newcommand{\CD}{{\mathcal{D}}}

\newcommand{\CF}{{\mathcal{F}}}

\newcommand{\prob}{{\bf P}}

\newcommand{\e}{{\bf E}}
\newcommand{\bae}{\begin{equation}\begin{aligned}}
\newcommand{\eae}{\end{aligned}\end{equation}}

\newcommand{\beq}{\begin{equation}}
\newcommand{\eeq}{\end{equation}}

\SHORTTITLE{Investigating the IF model as the limit of a random discharge model}

\TITLE{Investigating the integrate and fire model as the limit of a random discharge model: a stochastic analysis perspective} 

\AUTHORS{Jian-Guo Liu\footnote{Department of Mathematics and Department of Physics, Duke University.
    \EMAIL{jliu@phy.duke.edu}}
  \and 
 Ziheng Wang\footnote{Mathematical Institute, University of Oxford.
 	\EMAIL{wangz1@math.ox.ac.uk}}
 	\and
 	Yantong Xie\footnote{School of Mathematical Sciences, Peking University.
 	\EMAIL{1901110045@pku.edu.cn}}\and
 	Yuan Zhang\footnote{School of Mathematical Sciences, Peking University.
 	\EMAIL{zhangyuan@math.pku.edu.cn}}\and
 	Zhennan Zhou\footnote{Beijing International Center for Mathematical Research, Peking University.
 	\EMAIL{zhennan@bicmr.pku.edu.cn}}}

\SHORTAUTHOR{J.G.~Liu, Z.~Wang, Y.~Xie, Y.~Zhang, and Z.~Zhou}


\AMSSUBJ{60B10; 60H30; 92B20; 92C20} 

\SUBMITTED{February 19, 2021} 
\ACCEPTED{October 21, 2021} 

\VOLUME{1}
\YEAR{2021}
\PAPERNUM{2}
\DOI{10.46298/mna.7203}

\ABSTRACT{In the mean field integrate-and-fire model, the dynamics of a typical neuron within a large network is modeled as a diffusion-jump stochastic process whose jump takes place once the voltage reaches a threshold. In this work the main goal is to establish the convergence relationship between a regularized process and the original one where in the regularized process, the jump mechanism is replaced by a Poisson dynamic, and jump intensity within the classically forbidden domain goes to infinity as the regularization parameter vanishes. On the macroscopic level, the Fokker-Planck equation for the process with random discharges (i.e. Poisson jumps) is defined on the whole space, while the equation for the limit process is defined on the half space. However, using an iteration scheme the difficulty due to the domain differences has been greatly mitigated and the convergence for the stochastic process and the firing rates can be established. Moreover, we find polynomial-order convergence for the distribution by a re-normalization argument in probability theory. Finally, using numerical experiments we quantitatively explore the rate and the asymptotic behavior of convergence for both linear and nonlinear models.}

\begin{document}



\section{Introduction}


The classical description of the dynamics of a large set of neurons is based on deterministic/stochastic differential systems for the excitatory-inhibitory neuron network \cite{gerstner2002spiking,lapique1907recherches,Nykamp_thesis, tuckwell1988introduction}. One of the most famous models is the noisy leaky integrate and fire (LIF) model \cite{lapique1907recherches}, where the collective behavior of a neural network can be averaged as a self-consistent environment and within the network the typical behavior of a neuron is approximated by a stochastic process \cite{brunel2000dynamics, brunel1999fast,compte2000synaptic,Delarue2015,guillamon2004introduction,Inglis2015,mattia2002population,newhall2010cascade,newhall2010dynamics,renart2004mean,sirovich2006dynamics,Touboul2012} and the influence from the network is given by an average synaptic input by a mean-field approximation \cite{Delarue2015,Inglis2015, renart2004mean,Touboul2012}. In this model, the membrane potential of a typical neuron within the network is denoted by the state variable $X_t$. When the synaptic input of the network (denoted by $I(t)$) vanishes, the membrane potential relaxes to the resting value $V_L$. In the single neuron approximation, the synaptic input $I(t)$, which itself is another stochastic process, is replaced by a continuous-in-time counterpart $I_c (t)$ (see  e.g. \cite{brunel2000dynamics, brunel1999fast,mattia2002population,omurtag2000simulation,renart2004mean,sirovich2006dynamics}), which takes the drift-diffusion form
\beq
I \,dt \approx I_c \, dt= \mu_c \, dt+\sigma_c \,dB_t.
\eeq
Here, $B_t$ is standard Brownian motion, and in principle the two processes $I_c (t)$ and $I(t)$ have the same mean and variance. Thus between the firing events, the evolution of the membrane potential is given by the following stochastic differential equation
\beq
\label{ou with jump}
d X_t=(-X_t+V_L+\mu_c)\,dt+\sigma_c \,dB_t.
\eeq
Another important ingredient in the modeling is the firing-and-resetting mechanism: whenever the membrane voltage $X_t$ reaches a threshold value $V_F$, it is immediately reset to a specific value $V_R < V_F$. Namely,
\beq
\label{jump form}
X_t=V_R \quad \text{if} \quad X_{t^-}=V_F.
\eeq
The reader may refer to \cite{renart2004mean} for a thorough introduction to this subject. 

From the perspective of probability theory, the jump-diffusion processes of type \eqref{ou with jump} and \eqref{jump form} were first introduced and studied by Feller \cite{MR47886, MR63607} (in terms of transition semigroups), which apparently was not motivated by applications in neuroscience. More specifically, \cite{MR63607} named such a process an ``elementary return process" and presented its Fokker-Planck equation in a weak form, the proof of which was based on a Markov semigroup argument in \cite{MR47886}. In \cite{MR2353702,MR2499861, MR3244555, MR3016590}, the authors were concerned with the spectral properties of the generator of the stochastic process or related models, and showed the exponential convergence in time towards the stationary distribution. In particular, \cite{MR3244555} applied their results to a neuronal firing model driven by a Wiener process and computed the distribution of the first passage time. In the work \cite{MR573203,MR2673971}, the authors made assumptions on the stochastic process that were more relaxed than or modified from those in \cite{MR0346904} and proved the existence of a pathwise solution of such process in a generalized sense.

For the jump-diffusion process \eqref{ou with jump} and \eqref{jump form}, there has been a growing interest in studying the partial differential equation model for the dynamics of the probability density function with which the stochastic process $X_t$ is associated. From the heuristic viewpoint using It{\^o}'s calculus, it is widely accepted (\cite{caceres2011analysis, liu2020rigorous}) that the evolution of the probability density $f(x,t)\ge 0$ of finding neurons at voltage $x$ at time $t \ge 0$ satisfies the following Fokker-Planck equation on the half line with a singular source term 
\beq
\label{master equation0}
\frac{\partial f}{\partial t}(x,t)+\frac{\partial}{\partial x}[h f(x,t)]- a \frac{\partial^2 f}{\partial x^2} (x,t) = N(t) \delta (x-V_R),\quad x \in (- \infty, V_{F}), \quad t>0,
\eeq
where the drift velocity is $h = -x + V_L + \mu_c$, the diffusion coefficient is $a= \sigma_c^2/2$, $\delta(x)$ denotes the Dirac function and the precise definition of the mean firing rate $N(t)$ is given by \eqref{fire rate}. We complement \eqref{master equation0} with the following Dirichlet and initial boundary conditions:
\beq
f(V_F,t)=0, \quad f(-\infty, t)=0, \quad f(x,0)=f_{\text{in}}(x)\ge 0.
\eeq
Equation \eqref{master equation0} is supposed to be the evolution of a probability density, therefore
\beq
\label{conservation}
\int_{-\infty}^{V_F} f(x,t) \, dx = \int_{-\infty}^{V_F} f_{\text{in}}(x) \, d x=1.
\eeq
Due to the Dirichlet boundary condition at $x=V_F$, there is a time-dependent boundary flux escaping the domain, and a Dirac delta source term is added to the reset location $x=V_R$ to compensate for the loss.  It is straightforward to check that the conservation equation \eqref{conservation} characterizes the mean firing rate $N(t)$ as the flux of neurons at the firing voltage, which is implicitly given by
\beq
\label{fire rate}
N(t) := - a \frac{\partial f}{\partial x} (V_F, t) \ge 0.
\eeq

We remark that this delta function source term on the right hand side of \eqref{master equation0} is equivalent to setting the equation on $(-\infty, V_R) \cup (V_R, V_F)$ instead and imposing the following conditions
\[
f(V_R^-,t)=f(V_R^+,t), \quad  a\frac{\partial }{\partial x}f(V_R^-,t) - a\frac{\partial }{\partial x}f(V_R^+,t) = N(t), \quad  \forall t \ge 0.
\]
The firing events generate currents that propagate within the neuron network, which can be incorporated into this PDE model by expressing the drift velocity $h$ and the diffusion coefficient $a$ as functions of the mean-firing rate $N(t)$ (see e.g. \cite{caceres2011analysis,caceres2011numerical,carrillo2013classical,Yantong}). In the simplest form, the following choice has been widely considered 
	\beq
	\label{nonlinear}
	h(x,N(t))=-x+bN(t),\quad a(N(t))=a_0+a_1(N(t)).
	\eeq
	In particular, the term \(-x\) describes the leaky behavior and \(b\) models the connectivity of the network:
	\(b > 0\) describes excitatory networks and \(b < 0\) describes inhibitory networks. In Sections \ref{sec:pre} and \ref{model extension}, we only consider the simple case when \(b=a_1=0\) and \(a_0=1\), while in Section \ref{numerical} we numerically investigate the nonlinear case when \(b\neq 0\).

Many recent works are devoted to investigating the properties of solutions of such PDE models, including the finite-time blow-up of weak solutions, the multiplicity of steady solutions, the relative entropy estimate, the existence of classical solutions, \ structure-preserving numerical approximations, etc.  (see e.g. \cite{caceres2011analysis,caceres2011numerical,Caceres2014,Caceres2018,carrillo2013classical,Yantong} and the references therein.) However to the best of our knowledge, due to the firing-and-resetting mechanism in \eqref{ou with jump}, the rigorous derivation of the Fokker-Planck equation \eqref{master equation0} in the classical sense from the microscopic stochastic process had not yet been achieved by conventional methods.

In \cite{liu2020rigorous}, we established the rigorous connection between the linear Fokker-Planck equation \eqref{master equation0} and the microscopic stochastic model \eqref{ou with jump} and \eqref{jump form}. For simplicity, we assume for the rest of the work that
\beq 
\label{para_value}
 V_L=0,\quad V_R=0, \quad V_F=1,\quad \mu_c=0,\quad \sigma=\sqrt{2},
\eeq
which means the LIF model becomes a linear model with the interactions among the network neglected. With these assumptions, the process \eqref{ou with jump} and \eqref{jump form} becomes the standard O-U process with a ``hard wall" at $1$, i.e., whenever at time t, $X_t$ hits 1, it immediately jumps to $0$ and then we restart the O-U-like evolution independent of the past. Unlike the standard jump-diffusion process, the jumping time for $X_t$ is determined by the hard wall boundary and thus the classical It{\^ o} calculus is not directly applicable. Inspired by the renewal nature of $X_t$ that agrees with the pioneering work of Feller \cite{MR63607}, a novel strategy based on an iterated scheme was proposed in \cite{liu2020rigorous} to show that the probability density function (abbreviated by p.d.f.) of $X_t$ is the classical solution of its Fokker-Planck equation.

In fact, with the introduction of the auxiliary stochastic process counting the number of firing events, the density $f(x,t)$ of the potential $X_t$ can be decomposed as a summation of sub-density functions $\{f_n (x,t) \}_{n=0}^{\infty}$. Each sub-density naturally links to a less singular sub-PDE problem which is determined iteratively and of better regularity. Besides, with the exponential decay of decomposition all the regularities are preserved in the limit, and thus we concluded the desired properties for the PDE problem \eqref{master equation0}.

In this paper, we continue to study a family of related jump-diffusion processes \eqref{state 
dependent} parameterized by $\delta$, which are used to approximate the process $X_t$ as 
$\delta \to 0^+$. Introduced in \cite{Caceres2014}, such jump-diffusion processes are used 
to explore the reasonable modelling of the mean firing events on the macroscopic level, 
which are associated with the Fokker-Planck equations of neurons with random firing 
thresholds. As shown in \cite{caceres2011analysis}, the solution of the Fokker-Planck 
equation with a deterministic firing potential may blow up in finite time, which is speculated to 
be related to the synchronization activity of neuronal networks. The random discharge 
mechanism in \cite{Caceres2014} provides an alternative scenario of incorporating the 
synchronous states besides introducing time delay, refractory states, etc (see 
\cite{brunel2000dynamics, brunel1999fast, Caceres2014, cormier2020long, 
jahn2011motoneuron}). However, all the regularization models are mostly based on scientific 
intuition or technical insights of the PDE theory, among which the random discharge model is 
tractable on the microscopic level and is reminiscent of the well-studied kinetic equations on 
the macroscopic level. Thus, we choose to focus on the random discharge model to 
rationalize the regularization effect and show the convergence relationship between such a 
model and its limit as the regularization parameter vanishes.

We denote such processes by $X^\delta_t$, and the associated jumping rate $\lambda^\delta(X^\delta_t)=0$ when $X_t^\delta <1$, and $\lambda^\delta(X_t^\delta)=O(\delta^{-1})$ when $X_t^\delta>1+\delta$. Between the firing events, $X_t^\delta$ propagates along the O-U process
\beq 
\label{eq:OU2}
d X_t^\delta = - X_t^\delta \,d t + \sqrt 2 \,d B_t.
\eeq
Recall that there exists a ``hard wall" boundary for $X_t$, i.e. the firing event takes place 
whenever $X_t$ reaches 1. However, the jumps of $X_t^\delta$ are determined by a 
state-dependent Poisson measure, for it evolves as the standard O-U process when 
$X_t^\delta<1$ and can jump with a high rate once $X_t^\delta$ exceeds $1$, and such a 
jump process can be interpreted as a ``soft wall" boundary when $X_t^\delta \ge 1$. The 
precise definition of $X^\delta_t$ can be found in \eqref{continuerate} and \eqref{state 
dependent}. The Poisson jump model frequently appears in kinetic models 
\cite{Cercignani2001Rarefied, Dautray1988Mathematical, Desvillettes2011The, 
Erdos2015Linear} and with It{\^o}'s formula, we can derive its Dynkin's formula, forward and 
backward Kolmogorov equation and Feynman-Kac formula, etc. However, It{\^o} calculus is 
not directly applicable when $\delta=0$ and thus $X_t$ is seen as the singular limit of a 
family of regularized processes. The primary goal of this paper is to justify whether and in 
what sense the regularized model converges to the original one.

Formally speaking, $X_t$ can be seen as the limit of $X_t^\delta$ as $\delta \to 0^+$, for the distribution of the process $X_t^\delta$ is supposed to converge to the process $X_t$ as $\delta \rightarrow 0^+$. The rigorous justification of such convergence is challenging due to the domain differences, and the main contribution of this paper is to rigorously establish the relationship between the two processes. We first present the strong Feller property \eqref{feller} of the limit process $X_t$ by comparing it with the regular O-U process. Then by a similar iteration strategy as that in \cite{liu2020rigorous}, (see the decomposition \eqref{delta sub distribution} and the iteration relationship in Proposition \ref{iteration} for details) we get convergence for the marginal distribution. Utilizing the strong Feller property for $X_t$, we easily generalize the convergence result to any finite dimensional distribution by induction, which together with the relatively compactness of $\{X_t^\delta\}_{\delta>0}$ implies the weak convergence \eqref{weak convergence for process} for the processes. Also, by using the decomposition and iteration strategy, the weak convergence of the mean firing rate \eqref{firing rate convergence} can be established. Finally by a standard re-normalization method in probability theory, we rigorously prove a polynomial-order convergence rate \eqref{conv rate} for the marginal distribution.

As a complement and extension to the convergence justification, we numerically explore the convergence rate and asymptotic behavior of the process $X_t^\delta$ and its density function for both the linear cases considered in previous analysis and nonlinear cases  where the analysis is still intractable in the current framework. The numerical scheme is based on the Scharfetter-Gummel reformulation and preserves certain structures of the Fokker-Planck equation for \(X_t^\delta\)~\cite{Yantong}. We also find numerical evidence of the existence of an asymptotic profile. More specifically, the simulation results strongly suggest that one can numerically identify two parameters \(\alpha\) and \(\beta\)  such that
\begin{equation}
\label{eq:fdelta}
f^\delta(x,t)=\delta^\alpha\phi\left(\delta^\beta(x-1)\right)+o\left(\delta^\alpha\right),\quad \forall x\in[1,+\infty),
\end{equation}
where \(\phi\) is a profile function independent of \(\delta\). Equation \eqref{eq:fdelta} suggests how \(f^\delta\) vanishes on the half-space \(\{x\geq1\}\) in a self-similar fashion.

It is worth mentioning that some recent works \cite{cormier2020hopf, cormier2020long} are devoted to exploring the long time behavior of some non-linear McKean-Vlasov type SDEs arising from neuron models, where the firing mechanism follows the same random discharge dynamic as the one in this work but the diffusion process is absent. In particular, it is proved in \cite{cormier2020hopf} that the SDE admits time periodic solutions through a Hopf bifurcation analysis. Although in this paper we only focus on the linear case, the convergence result may provide tools and insights for studying  nonlinear models, especially for the long time behavior of the random discharge model. However, analysis of nonlinear models is beyond the scope of this paper. 

The rest of the paper is outlined as follows. In Section \ref{sec:pre}, we first review the jump 
process $X_t$ and give the precise definition of $X_t^\delta$. After laying out their 
Fokker-Planck equations, we summarize the main convergence results between $X_t^\delta$ 
and $X_t$. In Section~\ref{model extension} we rigorously show by a probabilistic approach 
that the distribution and mean firing rate of the random discharge model weakly converge to 
those of the original model. Also, a polynomial order convergence rate for the marginal 
distribution is established. In Section \ref{numerical}, we investigate using numerical 
experiments the rates and the asymptotic behavior of related convergence for both linear 
and nonlinear models. In Section \ref{conclusion}, we conclude this paper and give some 
future research directions. For the rest of this work, we use $\gamma$, $C$, $C_0$, $C_k$ 
and $C_T$ to denote generic constants.

\section{Preliminaries and main results} \label{sec:pre}

We first briefly review the results for the jump process $X_t$ in \cite{liu2020rigorous}. The stochastic process $X_t$ has been formally defined in \eqref{ou with jump} and \eqref{jump form}, and the interested readers may refer to \cite{liu2020rigorous} for the rigorous construction of such a process. In the integrate and fire model, the process $X_t$ is used to describe the mean field behavior of the neuron network. Let the distribution of $X_0$ be denoted by $\mu$. For technical reasons we suppose that $\mu$ is a probability measure on $(-\infty,1-\beta]$ for some $\beta>0$ and let $f_{\text{in}}(x)$ be its density. Then with the iteration idea, we have already shown in \cite{liu2020rigorous} that for any fixed $T>0$, the Fokker-Planck equation for $X_t$ is
\begin{equation}
\label{eq_planck}
\left\{
\begin{aligned}
&\frac{\partial f}{ \partial t} - \frac{\partial }{\partial x}\left( x f\right) - \frac{\partial ^2 f}{\partial x^2}=0, \quad x \in (-\infty,0) \cup (0,1), t\in (0,T],\\
&f(0^-,t)=f(0^+,t), \quad  \frac{\partial}{\partial x}f(0^-,t)- \frac{\partial}{\partial x}f(0^+,t)=-\frac{\partial}{\partial x}f(1^-,t), \quad t\in (0,T],\\
&f(-\infty,t)=0, \quad f(1,t)=0, \quad t\in [0,T],\\
&f(x,0)=f_{\text{in}}(x), \quad x\in (-\infty,1).
\end{aligned}
\right.
\end{equation}
In the rest of this paper, we define 
$$
N(t):=-\frac{\partial}{\partial x}f(1^-,t), \quad t>0,
$$
which serves as the definition of the mean firing rate. 

In this paper, we consider a related family of jump-diffusion processes parameterized by 
$\delta$ as in \eqref{eq:OU2}, which is related to the Fokker-Planck equation for neurons 
with random firing thresholds, to approximate the original process $X_t$. Introduced in 
\cite{Caceres2014}, we use the random discharge model to rationalize the regularization 
effect for $X_t$. The precise definition of $X_t^\delta$ is as follows.

For a fixed $\delta > 0$, we define the discharge rate function:
\begin{equation}\label{continuerate}
\lambda^{\delta}(x)=
\begin{cases}
0, & x\le 1 \\
\frac{x-1}{\delta^2}, &  x\in[1,1+\delta] \\
\frac{1}{\delta}, &  x\ge 1+\delta
\end{cases}
\end{equation}
Then consider the following state-dependent jump-diffusion process as defined in Chapter $\uppercase\expandafter{\romannumeral 6}$ of \cite{hanson2007applied}:
\begin{equation}
dX^{\delta}_t=-X^{\delta}_tdt+\sqrt 2 dB_t+[-X^{\delta}_{t_-}]dP^{\delta}(t,X^{\delta}_t)
\label{state dependent}
\end{equation}
where $P^{\delta}(t,X^{\delta}_t)$ is a Poisson point process with intensity $\lambda^{\delta}(X^{\delta}_t)$.
	\begin{remark}
	For the rest of this paper, we only consider the simplified initial condition for the processes, i.e. both $X_t$ and $X_t^\delta$ start from a fixed point $x\in (-\infty,1)$. In the following, $\e^x$ and $\prob^x$ denote the expectation and probability of a stochastic process starting from $x$. The natural extension to general and proper initial conditions can be obtained by convolution.
	\end{remark}

First by It{\^o}'s formula, we can directly derive the Fokker-Planck equation for $X_t^\delta$.
\begin{theorem}
	\label{FK for jump diffusion}
	Let $f^\delta(x,t)$ denote the p.d.f. of the process $X_t^\delta$ starting from $y<1$. It satisfies the following PDE problem in the sense of distributions
	\begin{equation}
	\label{eq:FKJD}
	\left\{
	\begin{aligned}
	\frac{\partial f^\delta}{ \partial t} - \frac{\partial }{\partial x}\left( x f^\delta \right) - \frac{\partial ^2 f^\delta}{\partial x^2}&=N^\delta (t) \delta(x)-\lambda^\delta(x)f^\delta(x,t), \quad x\in \mathbb R, \quad t>0, \\
	N^\delta (t) &=\int_{\mathbb R} \lambda^\delta(z)f^\delta(z,t)dz,\quad t>0, \\
	f^\delta(-\infty,t)&=f^\delta(+\infty,t)=0, \quad t>0,\\
	f^\delta(x,0)&=\delta(x-y) \ \mbox{in}\ \CD'(\mathbb R),
	\end{aligned}
	\right. 
	\end{equation}
	where $N^\delta(t)$ is the modified mean firing rate and  $\delta(x)$ denotes the Dirac function.
\end{theorem}
\begin{remark}
	To prevent confusion, we clarify that
	$\delta(x)$ is Dirac function and $\delta$ denotes the parameter in the intensity \eqref{continuerate}.
\end{remark}

Equation \eqref{eq:FKJD} is referred to as the Fokker-Planck equation for neuron networks with random discharges \cite{Caceres2014}. We observe that, in this variant model, the mean firing rate $N^\delta(t)$ is modified to an integral of the density function, which admits a global estimate as shown in \cite{Caceres2014}. The proof of Theorem \ref{FK for jump diffusion} is elementary, and we choose to omit the details in this work.

However for the process $X_t^\delta$, It{\^o}'s calculus is not directly applicable when $\delta=0$ and the approximation error between the two processes is not quantifiable. The main contribution of this paper is to rigorously establish the relationship between the two processes. We first prove the strong Feller property of the limit process $X_t$, which plays a key role in the proof of weak convergence. With the iteration strategy, we can prove weak convergence between the stochastic processes and the convergence rate of the marginal distribution, which is summarized in the following theorem.
\begin{theorem}
\label{relationship}
(\romannumeral1) (Strong Feller property.) For any bounded and Borel measurable function 
$\varphi$ and any $t>0$, $\e^x[\varphi(X_t)]$ is continuous against the starting point $x$. 
I.e., for any $\varphi\in \mathcal{B}_b(-\infty,1)$, $\epsilon>0$ and $x\in (-\infty,1)$, there 
exists $ \delta_0\in(0,\varepsilon)$ such that for any $y \in (-\infty,1)$ such that $ |y-x|\le 
\delta_0$, we have
\bae
\label{feller}
\left|\e^x[\varphi(X_t)]-\e^y[\varphi(X_t)]\right|< \varepsilon.
\eae
(\romannumeral2) (Weak convergence of stochastic process.) Let the stochastic processes $X_\cdot^\delta$ and $X_\cdot$ start from any fixed $x \in (-\infty,1)$. Without loss of generality, suppose that $X_0^\delta = X_0 = 0$. Then $X^\delta_\cdot$ converges weakly to $X_\cdot $ as $\delta \to 0^+$, i.e., for any bounded and continuous function $\varphi\in C_b(D_{\mathbb{R}}[0,\infty))$,
\beq
\label{weak convergence for process}
\lim_{\delta \to 0^+} \e^0 \varphi(X_\cdot^\delta) = \e^0 \varphi(X_\cdot),
\eeq
where $D_{\mathbb{R}}[0,\infty)$ denotes the space of right continuous functions $x: [0,\infty) \to \mathbb{R}$ with left limits.\\
(\romannumeral3) (Convergence for the mean firing rate.) The modified mean firing rate 
$N^\delta(t)$ also converges weakly to the mean firing rate $N(t)$, i.e. for any fixed $T>0$ 
and $\varphi\in C_b[0,T]$, 
\beq
\label{firing rate convergence}
\lim_{\delta\to 0^+} \int_0^T\varphi(t)N^\delta(t)dt=\int_0^T\varphi(t)N(t)dt.
\eeq
(\romannumeral4) (Convergence rate for marginal distribution.) Let $F^{\delta}(x,t)$ and $F(x,t)$ denote the cumulative distribution functions of $X_t^{\delta}$ and $X_t$ respectively. Without loss of generality, suppose that $X_0^\delta = X_0 = 0$. Then for any fixed $T<\infty$, there is a constant $\gamma \in (0,\infty)$ s.t. $\forall (x,t)\in\mathbb{R}\times[0,T]$, we have
\begin{equation}
	\label{conv rate}
	\left|F^{\delta}(x,t)-F(x,t)\right| = O(\delta^\gamma) \quad \text{as} \quad \delta \to 0^+.
\end{equation}
\end{theorem}
The detailed proof of the theorem is presented in Section \ref{model extension}, and is hereby outlined as follows. 
\begin{itemize}
		\item In Section \ref{feller section}, we show the strong Feller property \eqref{feller} by using the connection between the jump diffusion process $X_t$ and the standard O-U process.
		\item In Section \ref{1 dim distribution}, by the coupling method and the iteration approach in \cite{liu2020rigorous} we prove the convergence for the marginal distribution, i.e., for any fixed $T<\infty$ and $\forall$ $\epsilon>0$, there is a $\eta_0>0$, s.t. $\forall \delta<\eta_0$, we have
		\begin{equation}
			\label{conv}
			\left|F^{\delta}(x,t)-F(x,t)\right|\le \epsilon \quad \text{for} \quad (x,t)\in \mathbb{R}\times[0,T].
		\end{equation}
		\item Then in Section \ref{weak}, by the strong Feller property for $X_t$, we get convergence for any finite marginal distribution by induction, which together with the relatively compactness of $\{X_t^\delta\}_{\delta>0}$ gives the weak convergence for processes.
		\item In Section \ref{firing rate conv}, we prove weak convergence of the mean firing rate by the decomposition and iteration approach in \cite{liu2020rigorous}.
		\item Finally in Section \ref{convergence rate}, we rigorously prove a polynomial upper bound on the rate at which $F^\delta(x,t) \to F(x,t)$ as $\delta \to 0^+$ with a multiscale renormalization argument in probability theory.
	\end{itemize}

\section{The random discharge model and its convergence}\label{model extension}

In Section \ref{sec:pre}, we precisely defined a family of jump-diffusion processes $X_t^\delta$ that are associated with the Fokker-Planck equations with random discharges. With the jumping criterion slightly altered, we are able to derive the Fokker-Planck equation of $X_t^{\delta}$ by classical It{\^o}'s calculus, which is reckoned as a regularized model (see Theorem 6.1 in \cite{Caceres2014}). However, the rigorous justification of such convergence is challenging.

\subsection{Strong Feller property}\label{feller section}

First, we prove the strong Feller property of process $X_t$, which is useful in getting  convergence for the finite dimensional distributions. Dr. Lihu Xu at the University of Macau taught us the following nice and easy proof through direct communication.
\begin{proof}
	[Proof of the strong Feller property in Theorem \ref{relationship}:]
	Similarly to \cite{liu2020rigorous}, we can strictly construct the jump process $X_t$ starting from $x$ and let $T_1^x$ denote the first time it hits $1$. By the proof of Theorem $2$ in \cite{liu2020rigorous}, we know that  for any $\varepsilon_0>0$ the p.d.f $f_{T_1^y}(\cdot)$ of the first hitting time is uniformly continuous with respect to all $y<1-\varepsilon_0$. Thus for any fixed $t>0$, $\varphi\in C_b(-\infty,1)$ and $ \forall \varepsilon>0$, there exists $t_0\in(0,t)$ such that for $\forall y\le 1-\varepsilon$, 
	\beq
	\label{density y decay}
	\prob(T_1^y\le t_0)\le \frac{\varepsilon}{8(\|\varphi\|_{L^\infty}+1)}.
	\eeq 
	Recall that Skorohod \cite{MR0346904} has proved that $X_t$ is Markovian, thus 
	$$
	\phi_{\varphi,t-t_0}(x):=\e^x[\varphi(X_{t-t_0})]
	$$ 
	is clearly bounded and measurable against $x \in (-\infty,1)$. Now applying the strong Feller property of the regular OU process $\{OU_t\}_{t\ge 0}$ on $\phi_{\varphi,t-t_0}$, we have that for any $x\le 1-2\varepsilon$, there exists $\delta_0\in(0,\varepsilon)$ such that $\forall |y-x|\le \delta_0$,
	\beq
	\label{OU feller}
	\left|\e^x[\phi_{\varphi,t-t_0}(OU_{t_0})]-\e^y[\phi_{\varphi,t-t_0}(OU_{t_0})]\right| \le \frac{\varepsilon}{2}.
	\eeq
	Now to compare $\e^x[\varphi(X_t)]$ and $\e^y[\varphi(X_t)]$, we have 
	$$
	\e^x[\varphi(X_t)]=\e^x[\varphi(X_t)\mathbbm{1}_{T_1^x\le t_0}]+\e^x[\varphi(X_t)\mathbbm{1}_{T_1^x>t_0}]=:I^x_1+I^x_2.
	$$
	Using \eqref{density y decay}, we immediately have that $|I^x_1| \le \frac{\varepsilon}{8}$, and by the Markov property of $X_t$, we have 
	\bae
	I^x_2=&\e^x\left[\e^x\left[\varphi(X_t)\Big|X_{t_0}\right]\mathbbm{1}_{T_1^x> t_0}\right]\\
	=&\int_{-\infty}^1\e^x[\varphi(X_t)\Big|X_{t_0}=z]f_0^x(z,t_0)dz\\
	=&\int_{-\infty}^1\e^z[\varphi(X_{t-t_0})]f_0^x(z,t_0)dz\\
	=&\int_{-\infty}^1\phi_{\varphi,t-t_0}(z)f^x_{\text{ou}}(z,t_0)dz+\int_{-\infty}^1\phi_{\varphi,t-t_0}(z)\left[f_0^x(z,t_0)-f^x_{\text{ou}}(z,t_0)\right]dz\\
	=&:I^x_3+I^x_4,
	\eae
	where $f_0^x(y,t)dy = \prob^x(X_t \in dy, T_1^x>t)$ and $f_{\text{ou}}^x(y,t)$ denotes the p.d.f. of standard OU process starting from $x$. Noting that $f_0^x(y,t)$ is the p.d.f. for the killed OU process, then $f^x_{\text{ou}}(y,t) \ge f_0^x(y,t)$ and thus for $I_4^x$, we have
	\bae
	\label{bound for I^x_4}
	|I^x_4|&\le \|\varphi\|_{L^\infty}\int_{-\infty}^1 f^x_{\text{ou}}(z,t_0)-f_0^x(z,t_0)dz\\
	&\le \|\varphi\|_{L^\infty}\left[1-\int_{-\infty}^1f_0^x(z,t_0)dz\right]\\
	&= \|\varphi\|_{L^\infty}\prob^x(T_1^x\le t_0)\\
	&\le \frac{\varepsilon}{8}.
	\eae
	Similarly, for $y$ we have
	$$
	\e^y[\varphi(X_t)]=I_1^y+I_2^y
	$$
	\bae
	I^y_2=&\int_{-\infty}^1\phi_{\varphi,t-t_0}(z)f^y_{\text{ou}}(z,t_0)dz+\int_{-\infty}^1\phi_{\varphi,t-t_0}(z)[f_0^y(z,t_0)-f^y_{\text{ou}}(z,t_0)]dz\\
	=&:I^y_3+I^y_4.
	\eae
	Then $|I_1^y|\le\frac{\varepsilon}{8}$ by \eqref{density y decay} and with the same argument in \eqref{bound for I^x_4}, we have $|I^y_4|\le \frac{\varepsilon}{8}$ and thus 
	\beq
	\left|\e^x[\varphi(X_t)]-\e^y[\varphi(X_t)]\right| \le|I_1^x|+|I_1^y|+|I_4^x|+|I_4^y|+|I_3^x-I_3^y| \le \varepsilon.
	\eeq
	where the term $|I_3^x-I_3^y|$ is small because of \eqref{OU feller} and since $\varepsilon$ is arbitrary, the strong Feller property for $X_t$ is valid.
\end{proof}

\subsection{Weak Convergence}\label{process convergence}

\

\noindent

\subsubsection{Convergence of the marginal distribution}\label{1 dim distribution}

\

\noindent

Now we prove the marginal distribution convergence \eqref{conv}. For any $t>0$, let $F^{\delta}(x,t)=P(X^{\delta}_t \le x)$ denote the cumulative distribution function (abbreviated by c.d.f.) of $X^{\delta}_t$, and $F(x,t)$ and $f(x,t)$ are the c.d.f. and p.d.f. of $X_t$ respectively. In \cite{liu2020rigorous}, we let $n_t$ denote the counting process of jumping times of $X_t$ and $T_n$ be its $n$-th jumping time, with which we decompose $F(x,t)$ as the summation of sub-c.d.f. $F_n(x,t)= \prob\left(X_t\le x, n_t=n \right)$. Similarly for $X_t^\delta$, define
\begin{equation}
	n^{\delta}_t=\left|\lbrace{s:s\le t, X^{\delta}_s\ne X^{\delta}_{s-}\rbrace}\right|
\end{equation}
to be the counting process which denotes the number of jumping times before $t$. And for each $n\ge 1$, define the stopping times:
\begin{equation}
	\label{delta jumping time sequence}
	T^{\delta}_n=\inf\lbrace{t\ge 0: n^{\delta}_t=n\rbrace}.
\end{equation}
Let $F_{T^{\delta}_n}(t)$ and $f_{T^{\delta}_n}(t)$ be the c.d.f and p.d.f of $T^{\delta}_n$ respectively. Moreover, for each $n\ge 0$, we also define:
\begin{equation}
	\label{delta sub distribution}
	F^{\delta}_n(x,t)=\prob\left(X^{\delta}_t\le x, n^{\delta}_t=n \right).
\end{equation}
Using similar arguments as in section $2$ of \cite{liu2020rigorous}, we have the following relationship and the exponential decay of $F_n^\delta(x,t)$ with respect to $n$.
	\begin{proposition}
		\label{iteration}
		For all $n\ge 1$,
		\begin{eqnarray}
			F^{\delta}_n(x,t) &=& \int^t_0 F^{\delta}_{n-1}(x,t-s)dF_{T^{\delta}_1}(s)\\
			F_{T^{\delta}_n}(t) &=& \int^t_0 F_{T^{\delta}_{n-1}}(t-s)dF_{T^{\delta}_1}(s),
		\end{eqnarray}
		and there is a $\theta>0$ such that for any $T\in (0,\infty)$,
		\bae
		F^\delta_n(x, t)\le \exp(-\theta n+T)
		\label{eq_exp_decay}
		\eae
		for all $t\le T$ and $x\in (-\infty,1]$.
		\label{exp decay}
	\end{proposition}
	With the exponential decay of $F_n^\delta(x,t)$ with respect to $n$, we know that $F^\delta(x, t)$ is absolutely continuous with respect to the Lebesgue measure and use $f^\delta(x,t)$ to denote the p.d.f..
	
Before the discussion of technical details, we first outline the major steps as follows.
\begin{itemize}
\item[(\romannumeral 1)]
    We use the technique of coupling to compare the difference between $F^{\delta}_0(x,t)$ and $F_0(x,t)$ together with $F_{T^{\delta}_1}(t)$ and $F_{T_1}(t)$. 
 \item[(\romannumeral 2)]
    We prove the uniform continuity of $F_n(x,t)$ by the regularity of $F_0(x,t)$ and the iteration approach.
 \item[(\romannumeral 3)]
     With the uniform continuity, parallel to step (i), we estimate the difference between $F_n(x,t)$ and $F^{\delta}_n(x,t)$.
 \item[(\romannumeral 4)]
    With the exponential decay property of both $F_n(x,t)$ and $F^{\delta}_n(x,t)$, we complete the proof. In fact, we can decompose the difference between $F(x,t)$ and $F^\delta(x,t)$ into two terms. The first term is small because of the argument in step (\romannumeral 3), and the second term is small due to the exponential decay property.
\end{itemize}
We first state the following result for the c.d.f $F_{T_1}(t)$ for $T_1$; the proof can be found in section $2$ of \cite{liu2020rigorous}.
\begin{proposition}
\label{FT1}
For any fixed $T>0$, $F_{T_1}(\cdot)$ is uniformly continuous on $[0, T]$, i.e., $\forall \varepsilon >0$, there exists $\eta_1 = \eta_1(T)>0$, s.t. $\forall t,t'\in [0,T]$, $|t-t'|\le \eta_1$, we have
\begin{equation*}
\left|F_{T_1}(t)-F_{T_1}(t')\right|< \varepsilon.
\end{equation*}
\end{proposition}
Now we compare the difference between $F^{\delta}_0(x,t)$ and $F_0(x,t)$ together with $F_{T^{\delta}_1}(t)$ and $F_{T_1}(t)$.

\begin{lemma}
\label{F0 dif delta}
Fix any $T>0$ and for any $\varepsilon >0$, there is an $\eta_0>0$ such that for all  $\delta \in (0,\eta_0]$ and all $(x,t) \in \mathbb{R}\times[0,T]$,

\begin{equation}
\left|F^{\delta}_0(x,t)-F_0(x,t)\right|
<\varepsilon.
\label{0 difference}
\end{equation}
At the same time, we have:
\beq
\left|F_{T_1}(t)-F_{T^{\delta}_1}(t)\right|<\varepsilon.
\label{hitting difference}
\eeq
\end{lemma}

\begin{proof}
Noting that for any $y \in \BR$, 
\beq
\label{OU process}
Z^y_t = e^{-t}y+\sqrt{2} e^{-t}\int_0^te^sdB_s
\eeq
is an O-U process starting from $y$, we couple two stochastic processes $X_{t\land T_1}$ and $X^{\delta}_{t\land T^{\delta}_1}$ as follows:\\
(\romannumeral 1) Let $\{Z^0_t\}_{t\ge0}$ denote the standard O-U process starting from $0$ and let $\Gamma$ obey the exponential distribution $\exp(1)$ and be independent of the process $Z^0_t$.\\
(\romannumeral 2) Consider the following two stopping times:
\begin{equation*}
T^0_1=\inf \lbrace{t\ge0: Z^0_t=1\rbrace},\quad T^{0,\delta}_1=\inf \lbrace{t\ge0:\int^t_0 \lambda ^{\delta}(Z^0_s)ds=\Gamma\rbrace}.
\end{equation*}
By definition, $Z^0_{t\land T^0_1}$ and $Z^0_{t\land T^{0,\delta}_1}$ are identically distributed as $X_{t\land T_1}$ and $X^{\delta}_{t\land T^{\delta}_1}$, and at the same time we have $T^0_1\le T^{0,\delta}_1$. Thus one has
\begin{equation}
F_0(x,t)=\prob^0(Z^0_t\le x,T^0_1>t) \\
\le F^{\delta}_0(x,t)=\prob^0(Z^0_t\le x,T^{0,\delta}_1>t)
\end{equation}
while
\begin{equation}
F^{\delta}_0(x,t)-F_0(x,t)\le \prob^0(T^0_1\le t, T^{0,\delta}_1>t).
\label{F0d1}
\end{equation}
By the strong Markov property of the O-U process, if we restart $Z^0_t$ at $T^0_1$, then $\lbrace{Z^0_{s+T^0_1}\rbrace}_{s\ge0}$ forms a new O-U process starting at $1$ which is independent of $T^0_1$. Denote this process by $\hat{Z}_s$. Moreover, defining a new stopping time $\hat{T}_1^{0,\delta}=T^{0,\delta}_1-T^0_1$ with respect to $\{\hat{Z}_s\}_{s\ge0}$, one may have:
\begin{equation}
\prob^0(T^0_1\le t, T^{0,\delta}_1>t)=\int^t_0 \prob^1(\hat{T}_1^{0,\delta}>t-s)dF_{T^0_1}(s).
\end{equation}

As we have previously seen in Proposition \ref{FT1}, $F_{T_1}(s)$ is uniformly continuous on $[0,T]$ and note that $T_1$ and $T_1^0$ have the same distribution. Thus $\forall$ $\varepsilon>0$, $\exists$ $\eta_1 = \eta_1(T)>0$ s.t. $\forall s\ge0$, we have $F_{T_1}(s+\eta_1)-F_{T_1}(s)<\varepsilon$. Then for $t<\eta_1$, we have
\begin{equation*}
\prob^0(T^0_1\le t,T^{0,\delta}_1>t)\le \int^t_0 dF_{T_1}(s)\le F_{T_1}(\eta_1)< \epsilon.
\end{equation*}
For $t>\eta_1$, we have
\begin{align}
\begin{split}
\prob^0(T^0_1\le t,T^{0,\delta}_1>t)&=\int^t_0 \prob^1(\hat{T}_1^{0,\delta}>t-s)dF_{T_1}(s)\\
&=\int^{t-\eta_1}_0 \prob^1(\hat{T}_1^{0,\delta}>t-s)dF_{T_1}(s)+\int^t_{t-\eta_1}\prob^1(\hat{T}_1^{0,\delta}>t-s)dF_{T_1}(s)\\
&\le \prob^1(\hat{T}_1^{0,\delta}>\eta_1)+\left(F_{T_1}(t)-F_{T_1}(t-\eta_1)\right)\\
&\le \prob^1(\hat{T}_1^{0,\delta}>\eta_1)+\varepsilon.
\end{split}\label{F0d2}
\end{align}
Hence, it suffices to prove that for any fixed $\eta_1 >0$,

\begin{equation}
\lim_{\delta \to 0}\prob^1(\hat{T}_1^{0,\delta}>\eta_1)=0.\label{T10}
\end{equation}

For any $\eta_1 >0$ and $\delta >0$, consider the following random subset generated by $\hat{Z}_s$ which denotes the time that $\hat{Z}_s$ is above the level $1+\delta$ before $\eta_1$:
\begin{equation}
\hat{I}(\eta_1, \delta):=\{s< \eta_1: \hat{Z}_s> 1+\delta\}.
\end{equation}
By definition, one may see that
\bae
\label{calculations}
&\left\{\hat{T}_1^{0,\delta}\le \eta_1\right\} \supset \left\{\int_{T_1^0}^{T_1^0+\eta_1}\lambda^\delta(Z_s^0)ds\ge \Gamma\right\} \\
=& \left\{\int_{\{s<\eta_1: \hat{X}_s\in [1,1+\delta]\}}\lambda^\delta(\hat{Z}_s)ds+\int_{\{s<\eta_1: \hat{Z}_s> 1+\delta\}} \lambda^\delta(\hat{Z}_s)ds\ge \Gamma\right\}\\
\supset& \left\{\int_{\{s<\eta_1: \hat{Z}_s> 1+\delta\}} \lambda^\delta(\hat{Z}_s)ds\ge \Gamma\right\}=\left\{\frac{\mathfrak{L}\left(\hat{I}(\eta_1, \delta)\right)}{\delta}\ge \Gamma\right\}
\eae
Thus,
\begin{equation}
\label{compare}
\prob^1(\hat{T}_1^{0,\delta}\le \eta_1)\ge \prob^1\left(\frac{\mathfrak{L}\left(\hat{I}(\eta_1, \delta)\right)}{\delta}\ge \Gamma\right)
\end{equation}
where $\mathfrak{L}$ denotes the Lebesgue measure in $\mathbb{R}$. Recall that $\Gamma \sim \exp(1)$. Thus it suffices to prove that
$$
\lim_{\delta \to 0}\prob^1\left(\frac{\mathfrak{L}\left(\hat{I}(\eta_1, \delta)\right)}{\delta}\ge \Gamma\right)=1.
$$
First we consider the case $\delta=0$. With \eqref{OU process}, we know that
\begin{equation*}
\hat Z_t=e^{-t}+\sqrt 2 \int^t_0 e^{-(t-s)}dB_s
\end{equation*}
Moreover, by the pathwise continuity of $\hat{X}_t$, one may see that $\hat{I}(\eta_1, 0)$ is a.s. either an empty set or an nonempty open set. We first show it is a.s. nonempty. Letting $\widetilde{T_1}=\inf \lbrace{t>0: \hat Z_t>1\rbrace}$, it suffices to prove that
\beq
\label{aim}
\prob^1(\widetilde{T_1}=0)=1.
\eeq

The proof of \eqref{aim} relies on the $0 - 1$ Law for standard Brownian motion. See Theorem $7.2.3$ on Page $362$ of \cite{durrett2019probability} for details.\\
(\romannumeral 1) For any $\Delta t>0$, $\{\widetilde{T_1}<\Delta t\} \supset \{\hat Z_{\Delta t}>1\} = \{e^{-\Delta t}-\sqrt 2 \int^{\Delta t}_0 e^{-(\Delta t-s)}dB_s>1 \}\in \CF_{\Delta t}$, where $\CF$ is the natural filtration generated by ${\hat Z_s}$.\\
(\romannumeral 2) Thus for $\Delta t_n \to 0$, $\{\widetilde{T_1}=0\} \supset \{\hat X_{\Delta t_n}>1, \text{i.o.}\} \in \CF_{0+}^B$, where $\CF_{0+}^B$ is the infinitesimal increment $\sigma$-field of Brownian motion $\{B_t\}_{t\ge 0}$ and i.o. stands for infinitely often.\\
(\romannumeral 3) By the $0 - 1$ Law, we now only need to prove that $\prob^1(\hat Z_{\Delta t_n}>1, \text{i.o.})>0$.
At the same time, with \eqref{OU process} we have
\begin{equation}
\prob^1(\hat X_{\Delta t_n}>1)=\prob^1\left(N(0,1)>\frac{1-e^{\Delta t_n}}{ \sqrt{1-e^{-2\Delta t_n}}}\right).
\end{equation}
Now noting that $\frac{1-e^{\Delta t_n}}{ \sqrt{1-e^{-2\Delta t_n}}}=O(\Delta t_n)\to 0$ as $n\to \infty$, we have $\lim_{n\to \infty}\prob^1(\hat Z_{\Delta t_n}>1)=\frac{1}{2}>0$. Thus we have proved \eqref{aim} and hence
\begin{equation}
\prob^1(\mathfrak{L}(\hat{I}(\eta_1,0))>0)=1.
\end{equation}
Note that events $\{\mathfrak{L}(\hat{I}(\eta_1, \frac{1}{n})>\frac{1}{n}\Gamma\} \to \{\mathfrak{L}(\hat{I}(\eta_1,0))>0\}$ as $n\to$ $\infty$.
Thus for any $\varepsilon >0$, $\exists N$ s.t. for all $n\ge N$

\begin{equation}
\prob^1\left(\mathfrak{L}(\hat{I}(\delta_1,\frac{1}{n}))>\frac{1}{n} \Gamma\right)\ge 1-\varepsilon.
\end{equation}
Fixing any $\delta \le \frac{1}{N}$ and recalling that $\varepsilon$ is arbitrary, together with \eqref{compare}, we get \eqref{T10}. Combining \eqref{T10}, \eqref{F0d1} and \eqref{F0d2}, when $\delta$ is small we have 
$$
F^{\delta}_0(x,t)-F_0(x,t)\le \prob^0(T^0_1\le t, T^{0,\delta}_1>t)\le \prob^1(\hat{T}_1^{0,\delta}>\eta_1)+\varepsilon\le 2\varepsilon.
$$
and thus the proof of Lemma \ref{F0 dif delta} is complete.

\end{proof}

Before proceeding with the iterative argument, we need to iteratively derive the uniform continuity of $F_n(x,t)$. Recalling Proposition $3.1$ of \cite{liu2020rigorous}, we immediately get the following essential uniform continuity of $F_0(x,t)$.
\begin{proposition}
\label{F0 continuity}
Fix any $T>0$ and for the $\eta_1=\eta_1(T)>0$ in Proposition \ref{FT1}, there exists 
$\eta_2=\eta_2(T)\in(0,\eta_1)$ such that for all $x\in \mathbb{R}$ and $t,t'\in [\eta_1,T]$, 
$|t'-t|<\eta_2$, we have
\begin{equation}
\left|F_0(x,t)-F_0(x,t')\right|<\varepsilon.
\end{equation}\label{F0}
\end{proposition}

Next for $F_1(x,t)$, by Lemma $2.2$ of \cite{liu2020rigorous}, we have 
\begin{align*}
\begin{split}
F_1(x,t)&=\int^t_0 F_0(x,t-s)dF_{T_1}(s),\\
F_1(x,t')&=\int^{t'}_0 F_0(x,t'-s)dF_{T_1}(s).
\end{split}
\end{align*}
Now we prove:
\begin{corollary}
Fix any $T>0$ and recall the definition of $\eta_1, \eta_2$ in Proposition \ref{FT1}-\ref{F0}. For all $0\le t<t'<T$ such that $t'-t<\eta_2$, and any $x\in \mathbb{R}$, we always have
\begin{align}
\left|F_1(x,t)-F_1(x,t')\right|\le 3\epsilon.
\end{align}
\end{corollary}
\begin{remark}
Here we no longer need $t$, $t'$ to be away from 0.
\end{remark}
\begin{proof}
First, supposing $t\in [\eta_1,T]$, we may write:
\begin{equation*}
\begin{aligned}
F_1(x,t)&=\int^{t-\eta_1}_0F_0(x,t-s)dF_{T_1}(s)+\int^t_{t-\eta_1}F_0(x,t-s)dF_{T_1}(s)\\
&=: A_1+A_2,
\end{aligned}
\end{equation*}
while
\begin{equation*}
\begin{aligned}
F_1(x,t')&=\int^{t-\eta_1}_0 F_0(x,t'-s)dF_{T_1}(s)+\int^{t'}_{t-\eta_1}F_0(x,t'-s)dF_{T_1}(s)\\
&=:B_1+B_2.
\end{aligned}
\end{equation*}
Using the uniform continuity of $F_{T_1}(t)$ and since $F_0(x,t)\le 1$, we have 
\begin{equation*}
\left\{\begin{array}{l}
A_2\le \int^t_{t-\eta_1} dF_{T_1}(s)\le \varepsilon\\
B_2\le \int^{t'-\eta_1}_{t-\eta_1} dF_{T_1}(s)+\int^{t'}_{t'-\eta_1} dF_{T_1}(s)\le 2\varepsilon,
\end{array}\right.
\end{equation*}
which together with Proposition \ref{F0} imply that
\begin{equation*}
\begin{split}
\left|F_1(x,t)-F_1(x,t')\right|&\le |A_1-B_1|+|A_2-B_2|\\
&\le \int^{t-\eta_1}_0\left| F_0(x,t'-s)- F_0(x,t-s)\right|dF_{T_1}(s)+2\varepsilon \\
&\le \int^{t-\eta_1}_0 \varepsilon dF_{T_1}(s)+2\varepsilon \le 3\varepsilon.
\end{split}
\end{equation*}

When $t\le \eta_1$, we have
\begin{equation*}
\begin{split}
F_1(x,t)=&\int^t_0 F_0(x, t-s)dF_{T_1}(s)\le \int^{\eta_1}_0  dF_{T_1}(s)=F_{T_1}(\eta_1)\le \varepsilon.
\end{split}
\end{equation*}
And note that $\eta_2\le \eta_1$, while $t'-t \in (0,\eta_2)$
\begin{equation*}
F_1(x,t')\le \int^t_0 dF_{T_1}(s)+\int^{t'}_t dF_{T_1}(s)\le 2\varepsilon.
\end{equation*}
Thus $\left|F_1(x,t)-F_1(x,t')\right|\le 2\varepsilon$.

\end{proof}
Similarly, one may inductively prove:
\begin{corollary}
Fix any $T>0$ and for all $n\ge 1$, any $x\in \mathbb{R}$ and all $0\le t \le t'<T$ such that $t'-t<\eta_2$, we have
\begin{equation*}
\left|F_n(x,t)-F_n(x,t')\right|\le (2n+1)\varepsilon.
\end{equation*}\label{Fn}
\end{corollary}

Now with the uniform continuity of $F_n(x,t)$, we can continue the proof of Theorem \ref{relationship}. First, parallel to the proof of Lemma \ref{F0 dif delta}, we can consider the difference between $F_1(x,t)$ and $F^{\delta}_1(x,t)$. Specifically, we have the following lemma.
\begin{lemma}
Fix any $T>0$ and recall the definition of $\delta$ in Lemma \ref{F0 dif delta}. We have for any $(x,t)\in \mathbb{R}\times[0,T]$, $|F_1(x,t)-F^{\delta}_1(x,t)|\le 5\varepsilon$.\label{F1 dif delta}
\end{lemma}
\begin{proof}
Note that
\begin{equation*}
\begin{cases}
F_1(x,t)=\int^t_0 F_0(x,t-s) dF_{T_1}(s),\\
F_1^{\delta}(x,t)=\int^t_0 F_0^{\delta}(x,t-s)dF_{T^{\delta}_1}(s).
\end{cases}
\end{equation*}
For any $t\in [0,T]$, we introduce the intermediate term:
\begin{equation*}
\int^t_0 F_0(x,t-s) dF_{T^{\delta}_1}(s).
\end{equation*}
Recalling \eqref{0 difference} and \eqref{hitting difference} together with the fact that $F^{\delta}_0,F_0\in [0,1]$, we have 
\begin{equation}
\begin{aligned}
&\left|F_1(x,t)-F^{\delta}_1(x,t)\right|\\
\le & \int^t_0 \left|F^{\delta}_0(x,t-s)-F_0(x,t-s) \right|dF_{T^{\delta}_1}(s)+ \left| \int^t_0 F_0(x,t-s) d\left(F_{T_1}(s)-F_{T_1^{\delta}}(s)\right) \right| \\
\le & \varepsilon F_{T^{\delta}_1}(t)+\left| \e^0\left[F_0(x,t-T^0_1) \mathbbm{1}_{\lbrace{T^0_1\le t\rbrace}}\right]-\e^0\left[F_0(x,t-T_1^{0,\delta}) \mathbbm{1}_{\lbrace{T^{0,\delta}_1\le t\rbrace}}\right] \right|.
\end{aligned}
\label{F1 difference}
\end{equation}
Now recall the definition of $\eta_2$ in Proposition \ref{F0 continuity} and consider the following \lq\lq good event\rq\rq
\beq \label{eventG}
G:=\lbrace{T^0_1\le t- \eta_1- \eta_2, T^{0,\delta}_1-T^0_1 \le \eta_2\rbrace}.
\eeq
Recalling Proposition \ref{FT1}-\ref{F0 continuity} and \eqref{T10}, we have
\begin{equation}
\begin{aligned}
&\left| \e^0\left[F_0(x,t-T^0_1) \mathbbm{1}_{\lbrace{T^0_1\le t\rbrace}}\right]-\e^0\left[F_0(x,t-T_1^{0,\delta}) \mathbbm{1}_{\lbrace{T^{0,\delta}_1\le t\rbrace}}\right] \right|\\
\le & \e^0\left[\left|F_0(x,t-T^0_1)-F_0(x,t-T_1^{0,\delta})\right| \cdot \mathbbm{1}_G \right]+\prob^0(G^C\cap \{T^0_1\le t\})\\
\le & \varepsilon +\prob^0(T^0_1\in (t-\eta _1-\eta_2,t])+\prob^0(T^{0,\delta}_1-T^0_1>\eta_2)\le 4\varepsilon.
\end{aligned}
\end{equation}
Combining with \eqref{F1 difference}, we complete the proof.

\end{proof}

Now one may inductively prove
\begin{lemma}
\label{Fn dif delta}
Fix any $T>0$ and recall the definition of $\delta$ in Lemma \ref{F0 dif delta}. We have for all $n\ge 1$ and any $x\in \mathbb{R}\times[0,T]$, $|F_n(x,t)-F^{\delta}_n(x,t)|\le (n+2)^2\varepsilon$.
\end{lemma}
\begin{proof}
By Lemma \ref{F0 dif delta} and Lemma \ref{F1 dif delta}, the result has been shown to be true for $n=0$ and $1$. Now suppose Proposition 1 holds for all $k\le n-1$. Now for $k=n$, we have:
\begin{equation}
\left\{\begin{array}{l}
F_n(x,t)=\int^t_0 F_{n-1}(x,t-s)dF_{T_1}(s),\\
F_n^{\delta}(x,t)=\int^t_0 F^{\delta}_{n-1}(x,t-s)dF_{T_1^{\delta}}(s).
\end{array}\right.
\end{equation}
Again there is
\begin{equation}
\begin{split}
\left|F_n(x,t)-F^{\delta}_n(x,t)\right|&\le \int^t_0\left|F^{\delta}_{n-1}(x,t-s)-F_{n-1}(x,t-s)\right|dF_{T_1^{\delta}}(s)\\
&+\left|\e^0\left[F_{n-1}(x,t-T^0_1) \mathbbm{1}_{\{T^0_1\le t\}}\right]-\e^0\left[F_{n-1}(x,t-T_1^{0,\delta}) \mathbbm{1}_{\{T_1^{0,\delta}\le t \}}\right]\right|.
\end{split}
\end{equation}
Recall the ``good event" $G=\{T^0_1\le t-\eta_2, T^{0,\delta}_1-T^0_1 \le \eta_2\}$. We have
\begin{align}
\begin{split}
&\left| \e^0\left[F_{n-1}(x,t-T^0_1) \mathbbm{1}_{\{T^0_1\le t\}}\right]-\e^0\left[F_{n-1}(x,t-T_1^{0,\delta}) \mathbbm{1}_{\{T_1^{0,\delta}\le t \}}\right] \right|\\
\le & \e^0\left[\left|F_{n-1}(x,t-T^0_1)-F_{n-1}(x,t-T_1^{0,\delta})\right| \cdot \mathbbm{1}_G \right]+\prob^0(G^{C} \cap \{T^0_1\le t\})\\
\le &(2n-1)\varepsilon +\prob^0(T^0_1\in (t-\eta_2,t])+\prob^0(T^{0,\delta}_1-T^0_1>\eta_2)\\
\le &(2n+1)\varepsilon.
\end{split}
\end{align}
Thus we have
\begin{equation*}
\left|F_n(x,t)-F^{\delta}_n(x,t)\right|\le [(n+1)^2+(2n+1)]\varepsilon \le (n+2)^2\varepsilon.
\end{equation*}

\end{proof}

Finally, for all $T<\infty$, and any $t\in [0,T]$, $x\le 1$,
\begin{equation*}
F_n(x,t)\le \prob^0(T_n\le t),\quad F^{\delta}_n(x,t)\le \prob^0(T^{\delta}_n\le t).
\end{equation*}
By the argument in Lemma $2.3$ of \cite{liu2020rigorous}, we have already implied that $\exists$ a constant $C$ depending only on T such that
\begin{align}
\prob^0(T^{\delta}_n\le t)\le \prob^0(T_n\le t)\le \exp(-Cn).
\label{exp}
\end{align}
What's more, we can decompose the difference between $F(x,t)$ and $F^\delta(x,t)$ into two terms. That is,
\begin{equation}
\left|F(x,t)-F^\delta(x,t)\right|\le \sum_{i=0}^{n}\left|F_i(x,t)-F_i^\delta(x,t)\right|+\sum_{i=n+1}^{+\infty}\left|F_i(x,t)-F_i^\delta(x,t)\right|.
\label{total difference}
\end{equation}
Now using Lemma \ref{Fn dif delta} we know that the first term of \eqref{total difference} is small, while the second term is small due to the exponential decay property \eqref{exp}. Thus the proof of \eqref{conv} is complete.

\subsubsection{Weak convergence in the sense of process}\label{weak}

\

\noindent

Now we can prove weak convergence in the sense of stochastic process $X_\cdot^\delta \to X_\cdot $ as $\delta \to 0^+$. We first prove the convergence of the finite dimensional marginal distribution and then use the relative compactness to conclude weak convergence for process.

\begin{proposition}
	\label{finite dim convergence}
	For all integer $n>0$, $0<t_1<t_2<\cdots <t_n<\infty$ and $-\infty<a_i<b_i<1$, $i=1,2,\cdots,n$, we have 
	\beq
	\prob^0(X^\delta_{t_i}\in (a_i,b_i], i=1,2,\cdots, n) \to \prob^0(X_{t_i}\in (a_i,b_i], i=1,2,\cdots, n) \quad \text{as} \quad \delta \to 0^+,
	\eeq
	i.e.,
	\beq
	(X^\delta_{t_1}, X^\delta_{t_2},\cdots, X^\delta_{t_n}) \to (X_{t_1}, X_{t_2} \cdots, X_{t_n}).
	\eeq
	where $\prob^0(\cdot)$ denotes the process starts from $0$.
\end{proposition}
When $n=1$, the result for $X_t, X^\delta_t$ starting from $0$ has been proved in Section \ref{1 dim distribution}. By the same proof we have convergence  with respect to any $X_0^\delta=X_0=x_0$, where $x_0$ belongs to some compact subset of $(-\infty,1)$. Moreover, the $\delta$ can be chosen to be uniform, i.e., 
\begin{corollary}
	\label{uniform convergence in distribution}
	For any  $-\infty<a<b<1$, and any $x_0 \in [a,b]$, we let $F^{\delta,x_0}(x,t)$ and 
	$F^{x_0}(x,t)$ be the c.d.f. of $X_t^\delta$ and $X_t$ starting from $x_0$. Then for any 
	$T<\infty$ and $\epsilon>0$, there exists $\eta_0=\eta(T,a,b,\epsilon)>0$ s.t. for all 
	$0<\delta<\eta_0$, $x_0\in [a,b]$, $(x,t) \in \mathbb{R}\times[0,T]$, we have 
	\beq
	|F^{\delta,x_0}(x,t) - F^{x_0}(x,t)|\le \epsilon.
	\eeq
\end{corollary}
Thus we conclude the case for $n=1$ and then by the strong Feller  property for $X_t$ and induction, we can get Proposition \ref{finite dim convergence}; the detailed proof can be found in Appendix \ref{detailed proof}.
  
With the finite dimensional weak convergence shown as above, the rest of the proof of weak convergence for processes follows from a standard relative compactness argument. Note that $(\mathbb{R}, \mathcal{B})$ is clearly complete and separable. By Theorem 3.7.2 and Theorem 3.7.8 of \cite{ethier2009markov}, in order to show that $X_\cdot^\delta \to X_\cdot$, it suffices to prove the following:\\
(\romannumeral1) For any $T<\infty$ and $\epsilon>0$, $\exists$ a compact set $\Gamma \subset (-\infty, +\infty)$ such that for any $\delta>0$, we have 
\beq
\label{compact set}
\prob^0(X_t^\delta \in \Gamma, \forall t \in [0,T]) \ge 1-\epsilon.
\eeq
(\romannumeral2) For any $T<\infty$ and $\epsilon>0$, $\exists$ $\sigma > 0$ such that for any $\delta>0$, we have 
\beq
\label{continuity}
\prob^0(w'(X_\cdot^\delta, \sigma, T) \ge \epsilon) \le \epsilon.
\eeq
where 
$$
w'(X^\delta_\cdot, \sigma, T)=\inf_{\{t_i\}} \max_{i} \sup_{s,t \in [t_{i-1}, t_i)}|X^\delta_s-X^\delta_t|
$$
and $\{t_i\}$ ranges over all partitions of the form $0=t_0<t_1<\cdots<t_{n-1}<t_n=T$ with $t_i-t_{i-1}>\sigma$ for all $i$. See (6.2) in Chapter 3 of \cite{ethier2009markov} for details.

In order to verify the conditions above, we first need to recall the constructions in Section 2.2 of \cite{liu2020rigorous}. For claim (\romannumeral1), note that $T_n\le T_n^\delta$ for all $\delta>0$ and by exponential decay, we have for any fixed $T<\infty$ and $\epsilon>0$, $\exists$ $n_0$, s.t. 
$$
\prob^0(T^\delta_{n_0}\le T)\le \prob^0(T_{n_0}\le T) <\frac{\epsilon}{2}.
$$ 
Note that by Doob's inequality, we have that there is an $M_0<\infty$ s.t. 
$$
\prob^0(\max_{t\le T} |Y_t^{(i)}|>M_0) \le \prob^0(\max_{t\le T} |\int_0^t e^s dB_s|>M_0) <\frac{\epsilon}{2n_0}.
$$ 
Now consider $\Gamma=[-M_0, M_0]$ and event $A=\{T^\delta_{n_0}>T\} \cap \cup_{i=1}^{n_0}\{ \max_{t\le T} |Y_t^{(i)}|<M_0\}$, where
$\prob^0(A)>1-\frac{\epsilon}{2}-n_0 \cdot \frac{\epsilon}{2n_0} =1-\epsilon$. Then recalling that the trajectory of $X_t^\delta$ can be decomposed by parts of $Y_t^{(i)}, t\in [0,T]$, thus in event $A$ for any $\delta>0, t \in [0,T]$, $X_t^\delta$ belongs to $\Gamma$, which gives \eqref{compact set}.

Finally, in order to check claim (\romannumeral2), we first note that by Proposition 3.1 of \cite{liu2020rigorous}, $f_{T_1}(\cdot) \in C[0,T]$ with $f_{T_1}(0)=0$. Thus for all $\epsilon>0$, $\exists$ $n_0<\infty$, s.t. $f_{T_1}(t) < \frac{\epsilon}{2T}$ for all $t \le \frac{T}{n_0}$. Then consider the event $B_1=\cap_{i=1}^{n_0} \{\tau_i> \frac{T}{n_0}\}$, where 
\beq
\prob^0(B_1^c) \le  n_0 \prob^0(\tau_1 \le \frac{T}{n_0}) \le n_0 \cdot \int_0^{n_0} \frac{\epsilon}{2T} dt <\frac{\epsilon}{2}.
\eeq

Moreover, note that the OU-process $Y_t^{(i)}$ is a.s. uniformly continuous. Thus $\exists$ $\sigma_1>0$ s.t. for each $i$, 
\beq
\prob^0(\exists s<t \in [0,T], t-s< \sigma_1, |Y_t^{(i)}-Y_s^{(i)}| \ge \epsilon) <\frac{\epsilon}{2n_0}.
\eeq
Then consider event 
$$
B_2=\cap_{i=1}^{n_0}(\forall s<t \in [0,T], t-s< \sigma_1, |Y_t^{(i)}-Y_s^{(i)}| < \epsilon),
$$
and $\prob^0(B_2^c)\le n_0\cdot \frac{\epsilon}{2n_0} \le \frac{\epsilon}{2}$. Thus let $\sigma=\min\{ \frac{\sigma_1}{3}, \frac{T}{2n_0} \}$ and $B=B_1\cap B_2$. Then in the event $B$, the jump-diffusion process $\{X_t^\delta\}_{t\in [0,T]}$ is composed of at most $n_0$ uniformly continuous O-U process each with length at least $\frac{T}{n_0}\ge \sigma$. Thus one may always construct a partition with 
$$
\max_{i}\sup_{s,t \in [t_{i-1},t_i)} |X^\delta(s) -X^\delta(t)|< \epsilon.
$$
So we have 
$$
\prob^0(w'(X_\cdot^\delta, \sigma, T) \le \epsilon) \ge \prob^0(B) > 1-\epsilon,
$$
which gives \eqref{continuity} and thus the proof of the weak convergence for process is complete.

\subsection{Weak Convergence for the Mean Firing Rate}\label{firing rate conv}

\

\noindent

In this section we consider the convergence of the mean firing rate by the iteration approach. Clearly, the density function $f^\delta(x,t)$ for the jump-diffusion process $X_t^\delta$ in \eqref{state dependent} and the mean firing rate $N^\delta(t)=\int_{\mathbb R} f^\delta(y,t) \lambda^\delta(y) dy$ admit the following expansions
\beq
f^\delta(x,t) = \sum_{n=0}^\infty f^\delta_n(x,t), \quad N^\delta(t)= \sum_{n=0}^{\infty} N^\delta_n(t).
\eeq
where  $ f_n^\delta(x,t) $ is the density function of the measure induced by $F_n^\delta(\cdot,t)$ in \eqref{delta sub distribution} and for $n\ge 0$,
\[
N^\delta_n (t):=\int_{\mathbb R} f^\delta_{n}(y,t) \lambda^\delta(y) dy.
\]
To prove \eqref{firing rate convergence}, we need to build a connection between $N^\delta_n(t)$ and the p.d.f. $f_{T^\delta_n}(t)$ of the jumping time in \eqref{delta jumping time sequence}. We first derive the Dynkin's formula for the killed process $\widetilde{X^\delta_t}$ that is obtained by stopping the process $X_t^\delta$ at the first jumping time $T_1^\delta$. To be specific, 
\beq
\label{OUkilling}
\widetilde{X^\delta_t}= 
\begin{cases}X_t^\delta,\quad \ \ & 
t<T_1^\delta,\\ 
X_{T_1^\delta}^\delta, \quad \ \ & t\ge T_1^\delta, 
\end{cases} 
\eeq

First, we derive the Fokker-Planck equation for $f_0^\delta$ and its decay property for further iteration calculations.
\begin{lemma}
Let $f_0^\delta(x,t) $ be the density of the measure induced by $F_0^\delta(\cdot,t)$ defined in \eqref{delta sub distribution}. Then it is the classical solution of the following equation
\begin{equation}
\label{f0delta equation}
\left\{
\begin{aligned}
\frac{\partial f^\delta_0}{ \partial t} - \frac{\partial }{\partial x}\left( x f^\delta_0 \right) - \frac{\partial ^2 f^\delta_0}{\partial x^2}&=-\lambda^\delta(x)f^\delta_0 (x,t), \quad x\in \mathbb R, \quad t>0,
\\
f^\delta_0(x,0)& = \delta(x)\  \mbox{in}\ \CD'(-\infty, +\infty),
\end{aligned}
\right.
\end{equation}
where $\delta(x)$ denotes the Dirac function. Moreover, for any $t>0$, and $|x|$ sufficiently large, one should have $\exists$ $C>0$ s.t.
\beq
\label{f0delta estimation}
\max\left\{\left|f_0^\delta(x,t)\right|, \quad \left|\frac{\partial}{\partial x}f_0^\delta(x,t)\right|, \quad \left|\frac{\partial^2}{\partial x^2}f_0^\delta(x,t)\right|\right\}\le \exp(-Cx^2),
\eeq
\end{lemma}
\begin{remark}
	The proof follows the standard argument as in \cite{delarue2013first}.
\end{remark}
\begin{proof}
    Recall in \eqref{OU process} that we use $Z^0_t$ to denote an O-U process starting from $0$. The proof of \eqref{f0delta estimation} follows the same argument as in Lemma $3.1$ of \cite{liu2020rigorous}. For any fixed $T>0$, according to Theorem $3.5$ in Chapter $\uppercase\expandafter{\romannumeral 5}$ of \cite{garroni1992green} by Garroni and Menaldi, there exists a unique Green's function 
	$G: \mathbb R\times[0,T]\times \mathbb R\times[0,T] \ni (y,s,x,t)\mapsto G(y,s,x,t)$ for the parabolic operator
	$$
	\mathcal L_y=-y\partial_y\cdot+\partial^2_{yy}\cdot-\lambda^\delta\cdot ,
	$$ 
	That is for a given $(x,t)\in \mathbb R\times[0,T]$, the function $\mathbb R\times[0,t)\ni (y,s)\mapsto G(y,s,x,t)$ is a solution of the PDE
	\begin{equation}
	\left\{
	\begin{aligned}
	&\partial_s G(y,s,x,t)+\mathcal L_y G(y,s,x,t)=0, \quad (y,s) \in \mathbb R\times[0,t),\\
	&G(y,t,x,t)= \delta(y-x)\  \mbox{in}\ \CD'(\mathbb R)
	\end{aligned}
	\right.
	\end{equation}	
	Following Theorem 5 in Chap. 9 of  \cite{friedman2008partial}, for any given $(y,s)\in \mathbb R\times [0,T)$, the function $\mathbb R\times(s,T] \ni (x,t)\mapsto G(y,s,x,t)$ is also known to be Green's function of the adjoint operator 
	$$
	\mathcal{L}_x^*=\partial_x[x\cdot]+\partial^2_{xx} \cdot-\lambda^\delta\cdot
	$$
	i.e. 
	\begin{equation}
	\label{first special equation delta}
	\left\{
	\begin{aligned}
	\partial_t G(y,s,x,t)&=\mathcal{L}_x^* G(y,s,x,t), \quad (x,t)\in \mathbb R\times(s,T],\\
	G(y,s,x,s)&= \delta(x-y)\  \mbox{in}\ \CD'(\mathbb R),
	\end{aligned}
	\right.
	\end{equation}
	Morever, it belongs to $C^{2,1}$ in $x,t$ and satisfies the following estimate:
	\begin{equation}
	\label{Gest}
	\left|\partial^\ell G(y,s,x,t)\right|\le C(t-s)^{-\frac{1+\ell}{2}}\exp\left(-C_0\frac{(x-y)^2}{t-s}\right),\quad 0\le s<t\le T.
	\end{equation}
	where $\ell=0,1,2$, $\partial^{\ell}=\partial_{x t}^{\ell}=\partial_{x}^{m} \partial_{t}^{n}, \,\ell=2 m+n,$ for $m,n \in \mathbb N_0$.
	
	Thus given any smooth test function $\phi: \mathbb R\times [0,T]$ with compact support, the PDE problem
	\begin{equation}
	\label{second special equation delta}
	\left\{
	\begin{aligned}
	\partial_s u(y,s) &= y\partial_y u(y,s)-\partial_{yy} u(y,s) +\lambda^\delta(y)u(y,s)-\phi(y,s), \quad (y,s)\in \mathbb R\times[0,T),\\
	u(y,T)&=0 \quad y\in\mathbb R
	\end{aligned}
	\right.
	\end{equation}
    admits a unique classical solution
	\beq
	\label{convolution}
	u(y,s)=\int_s^T\int_{-\infty}^{+\infty}G(y,s,x,t)\phi(x,t)dxdt.
	\eeq
    Set 
	$$
	M_t:=u( Z^0_t,t) \quad \text{and} \quad N_t:=\exp\{-\int_0^t\lambda^\delta(Z^0_s)ds\}
	$$ 
	--- they are both semimartingales. Then by It{\^o}'s formula (see Exercise  $5.32$ on Page $209$ of \cite{Liggett2010Continuous} for details), we have
	\beq
	\label{ito for YZ}
	d(M_t N_t)=M_tdN_t+N_tdM_t+d\langle M,N\rangle_t.
	\eeq
	Note that $dN_t=-N_t\lambda^\delta(Z^0_t)dt$, thus $N_t$ is of bounded variation and then the quadratic variation $\langle N\rangle_t=0$. By $\langle M,N\rangle_t\le \langle M\rangle_t\langle N\rangle_t$, we know the covariance process $\langle M,N\rangle_t$ is equal to $0$. Hence
	\bae
	\label{ito calculus}
	d(M_tN_t)=&-u(Z^0_t,t)\exp\{-\int_0^t\lambda^\delta(Z^0_s)ds\}\lambda^\delta( Z^0_t)dt + \exp\{-\int_0^t\lambda^\delta(Z^0_s)ds\}\\
	&\cdot\left(\left[u_t(Z^0_t,t)-u_x(Z^0_t,t)Z^0_t+u_{xx}( Z^0_t,t)\right]dt+\sqrt{2}u_x(Z^0_t,t)dB_t\right).
	\eae
	With \eqref{ito calculus} and the boundary condition of $u$ in \eqref{second special 
	equation delta}, we have 
	\bae
	\label{Ito calculus}
	0&=u( Z^0_T, T)\exp(-\int_0^T\lambda^\delta(Z^0_t)dt)\\
	&=u(0,0)-\int_0^Tu(Z^0_t, t)\lambda^\delta( Z^0_t)\exp(-\int_0^t\lambda^\delta(Z^0_s)ds)dt\\
	&+\int_0^T\left[u_t(Z^0_t, t)-xu_x(Z^0_t, t)+u_{xx}(Z^0_t, t)\right]\exp(-\int_0^t\lambda^\delta(Z^0_s)ds)dt\\
	&+\sqrt{2}\int_0^Tu_x(Z^0_t, t)\exp(-\int_0^t\lambda^\delta(Z^0_s)ds)dB_t.
	\eae
	Taking the expectation of \eqref{Ito calculus} and recalling \eqref{second special equation 
	delta}, we have 
	\bae
	u(0,0)&=\e^0\left[\int_0^T\phi(Z^0_t, t)\exp(-\int_0^t\lambda^\delta(Z^0_s)ds)dt\right]\\
	&=\int_0^T\e^0\left[\phi(Z^0_t, t)\exp(-\int_0^t\lambda^\delta(Z^0_s)ds)\right]dt
	\eae
	Now applying formula $(8.2.10)$ on Page $139$ of \cite{Oksendal1998Stochastic} with $f(\cdot)=\phi(\cdot, t)$, 
	\beq
	\e^0\left[\phi(Z^0_t, t)\exp(-\int_0^t\lambda^\delta(Z^0_s)ds)\right]=\e^0\left[\phi(X^\delta_t, t)\mathbbm{1}_{\{T_1^\delta>t\}}\right]
	=\int_{-\infty}^{+\infty}\phi(x,t)f_0^\delta(x,t)dx.
	\eeq
	By \eqref{convolution} we have 
	$$
	\int_0^T\int_{-\infty}^{+\infty}\phi(x,t)G(0,0,x,t)dxdt=u(0,0)=\int_0^T\int_{-\infty}^{+\infty}\phi(x,t)f_0^\delta(x,t) dxdt,
	$$
	which implies that $G(0,0,x,t)$ coincides with $f_0^\delta(x,t)$. Thus we conclude that $f_0^\delta(x,t)$ satisfies \eqref{f0delta equation} and the decay property \eqref{f0delta estimation} is valid because of \eqref{Gest}. 
	
\end{proof}
Next we can prove 
\begin{lemma}
	For any $n\ge 1$ and $t>0$,
	\beq
	\label{stopping time relationship}
	f_{T_n^\delta}(t)=N^\delta_{n-1}(t).
	\eeq
\end{lemma}
\begin{proof}
We prove \eqref{stopping time relationship} inductively. First for the case when $n=1$, with \eqref{f0delta equation} and \eqref{f0delta estimation}, one has for any $t>0$,
\begin{equation}
\begin{aligned}
\label{stopping time relationship 1}
f_{T^\delta_1}(t)=&\frac{d}{dt}\prob^0(T_1^\delta\le t)=-\frac{d}{dt}\int_{-\infty}^{+\infty}f_0^\delta(x,t)dx\\
=&-\int_{-\infty}^{+\infty}\frac{d}{dt}f_0^\delta(x,t)dx\\
=&\int_{-\infty}^{+\infty} \left[\lambda^\delta(x)f^\delta_0(x,t)-\frac{\partial ^2 f^\delta_0}{\partial x^2}-\frac{\partial}{\partial x}\left( x f^\delta_0 \right)\right]dx\\
=&\int_{-\infty}^{+\infty} \lambda^\delta(x)f^\delta_0(x,t)dx.
\end{aligned}
\end{equation}
Now we assume that \eqref{stopping time relationship} holds for all $k\le n$ and note that 
$$
f_{T^\delta_{n+1}}(t)=\int_0^tf_{T_n^\delta}(t-s)f_{T_1^\delta}(s)ds
=\int_0^t\int_{-\infty}^{+\infty}f^\delta_{n-1}(x,t-s)\lambda^\delta(x)dxf_{T_1^\delta}(s)ds.
$$
By Fubini's formula,
\beq
f_{T^\delta_{n+1}}(t)=\int_{-\infty}^{+\infty}\int_0^tf^\delta_{n-1}(x,t-s)f_{T_1^\delta}(s)ds\lambda^\delta(x)dx=\int_{-\infty}^{+\infty}f_n^\delta(x,t)\lambda^\delta(x)dx=N_n(t).
\eeq

\end{proof}

Now we can show the weak convergence of $N^\delta$ as $\delta \rightarrow 0^+$. More precisely,  given $T>0$, for any smooth test function $\phi(t)\in C_b[0,T]$, we have
\beq
\label{conv jump}
\left|\int_0^T\phi(t)N^\delta(t)dt-\int_0^T\phi(t)N(t)dt \right| \rightarrow 0, \quad \text{as} \quad \delta \rightarrow 0^+.
\eeq
Notice that both $N(t)$ and $N^\delta(t)$ have series representations.
In light of the following decomposition
\begin{equation*}
\begin{aligned}
&\left|\int_0^T\phi(t)N^\delta(t)dt-\int_0^T\phi(t)N(t)dt \right|\\
\le &\sum_{i=1}^{N}\left|\int_0^T\phi(t)F'_{T^{\delta}_n}(t)dt-\int_0^T\phi(t)F'_{T_n}(t)dt \right|+\sum_{i=N+1}^{+\infty}\left|\int_0^T\phi(t)F'_{T^{\delta}_n}(t)dt-\int_0^T\phi(t)F'_{T_n}(t)dt \right|\\
=&:I_1+I_2.
\end{aligned}
\end{equation*}
Due to the exponential decay property \eqref{exp}, $I_2\to 0$ as $\delta \to 0^+$ and thus it suffices to estimate $I_1$, noting that $\phi$ is bounded, 
$$
I_1=\sum_{i=1}^{N} \left|\int_0^T\phi(t)dF_{T^{\delta}_n}(t)-\int_0^T\phi(t)dF_{T_n}(t) \right|
$$
Noting that $T_1^\delta \to T_1$ in distribution and $T_n, T_n^\delta$ are the i.i.d. summation of $T_1, T_1^\delta$ respectively, we have  $T_n^\delta \to T_n$ in distribution as $\delta \to 0^+$ and thus for any $n$,
$$
\lim_{\delta\to 0^+} \left|\int_0^T\phi(t)dF_{T^{\delta}_n}(t)-\int_0^T\phi(t)dF_{T_n}(t) \right|=0.
$$
Noting that the summation in $I_1$ is finite, we conclude that $I_1\to 0$ as $\delta \to 0^+$. Hence, the proof for the weak convergence of $N^\delta$ is complete.

\subsection{Convergence rate}\label{convergence rate}
\

\noindent

Finally, we aim to rigorously prove a polynomial-order convergence rate of $F^\delta(x,t) \to F(x,t)$ as $\delta \to 0^+$ as in \eqref{conv rate}. For any sufficiently small $\epsilon>0$, let $\epsilon_0=\epsilon^2$ and $n_0=\frac{1}{C_T}\log(\epsilon_0^{-1})$. By lemma \ref{F0 dif delta} and Lemma \ref{Fn dif delta}, for all $\epsilon_0>0$, $\exists$ $\eta_0 = \eta_0(\epsilon_0)>0$ s.t. $\forall \delta \le \eta_0$, $n\ge 0$ and $(x,t) \in \mathbb{R}\times [0,T]$,
$$
\left| F_n^\delta(x,t) - F_n(x,t) \right| \le (n+2)^2\epsilon_0.
$$   
Moreover, recalling \eqref{exp}, for any fixed $T<\infty$, there $\exists$ $C_T>0$ s.t. 
$$
\prob^0(T^{\delta}_n\le t)\le \prob^0(T_n\le t)\le \exp(-C_Tn).
$$
Then for all $\delta < \eta_0(\epsilon_0)$, $t \in [0,T]$ and $x \in \mathbb{R}$, 
\bae
&|F^\delta(x,t) - F(x,t)|\\
\le& \sum_{n=0}^{n_0-1} |F_n^\delta(x,t) - F_n(x,t)| + 2 \prob^0(T_{n_0}\le t)\\
\le& (n_0+2)^3 \epsilon_0 + 2\epsilon_0< \epsilon.
\eae

Thus to get the convergence rate of a polynomial-order, we only need to find a lower bound for $\eta_0$, which is polynomial with respect to $\epsilon$. In the order to prove the result of interest, it suffices to show the following polynomial-order relationship:
\begin{itemize}
	\item[(\romannumeral1)] The $\eta_1$ and $\eta_2$ defined in Proposition \ref{FT1} and Proposition \ref{F0 continuity} are both of a polynomial-order of $\epsilon$.
	\item[(\romannumeral2)] The $\eta_0$ in Lemma \ref{F0 dif delta} is of a polynomial order of $\eta_1$ and $\eta_2$, and thus also a polynomial-order of $\epsilon$.
\end{itemize}

\begin{proof}
	[Proof of (\romannumeral1):] For $\eta_1$, we can always set $\eta_1 = \frac{\epsilon}{\max_{t \le T} f_{T_1}(t)+1}$ by Proposition \ref{FT1}. As for $\eta_2$, when $\epsilon<1$ is sufficiently small, one may let $\eta_2=\eta_1^3= \frac{\epsilon^3}{[\max_{t \le T} f_{T_1}(t)+1]^3}$ and prove that it satisfies the condition in Proposition \ref{F0 continuity}. For any $\eta_1<t'<t\le T$, $t-t'\le \eta_2$, we have that for any $x\in \mathbb{R}$,
	\beq
	\left|F_0(x,t) - F_0(x,t')\right| = \left|\int_{-\infty}^x \left[f_0(y,t)-f_0(y,t')\right]dy \right| \le \int_{-\infty}^x \int_{t'}^t \left| \frac{df_0}{ds}(y,s) \right| dsdy. 
	\eeq
	By the estimation of the Green's function (38) of \cite{liu2020rigorous}, $\exists$ $C, C_0<\infty$ that depend only on $T$ s.t. for any $s\in [t',t]$,
	$$
	\left| \frac{df_0}{dt}(y,s) \right| \le C s^{-1} \exp\left( -C_0\frac{y^2}{s} \right) \le C \eta_1^{-1} \exp \left( -C_0\frac{y^2}{T} \right).
	$$
	Thus 
	\beq
	|F_0(x,t) - F_0(x,t')| \le \eta_1^2 \cdot C \int_{-\infty}^1 \exp \left( -C_0\frac{y^2}{T} \right) dy <\epsilon.
	\eeq
\end{proof}

For (\romannumeral2), recalling the proof of Lemma \ref{F0 dif delta}, the choice of $\eta_0$ is decided by \eqref{T10}, thus with the following lemma we can find a polynomial order of $\eta_1$ as a lower bound for $\eta_0$.
\begin{proposition}
	\label{polynomial decay}
	For all sufficiently small $\epsilon>0$, $\exists$ $\gamma\in (0,\infty)$ such that $\forall \delta<\epsilon^\gamma$, we have 
	\beq
	\label{poly rate}
	\prob^1(\hat{T}_1^{0,\delta} > \epsilon) <\epsilon.
	\eeq
\end{proposition}
In order to prove \eqref{poly rate}, recalling \eqref{calculations} we have  
$$
\left\{\hat{T}_1^{0,\delta} \le \epsilon \right\} \supset \left\{\frac{\mathfrak{L}\left(\hat{I}(\epsilon, \delta)\right)}{\delta} \ge \Gamma\right\}
$$
Thus it suffices to show that for any $\delta< \epsilon^\gamma$
\bae
\label{speed aim}
\prob^1\left( \frac{\mathfrak{L}\left(\hat{I}(\epsilon, \delta)\right)}{\delta} \ge \Gamma \right) = \prob^1\left( \frac{\int_0^{\epsilon} \mathbbm{1}_{ \{Z^1_t > 1+\delta \} }dt}{\delta} \ge \Gamma \right)\ge 1-\epsilon,
\eae
where $Z^1_t$ denotes a standard O-U process starting from $1$ as in \eqref{OU process} where $\Gamma$ is an independent exponential distribution obeying $\exp(1)$.

Now we use a renormalization argument which is standard for Brownian Motion (abbreviated by B.M.) to prove \eqref{speed aim}. We first introduce a sequence of scales as follows: define $\epsilon_n =\epsilon \cdot 2^{-n}$ for all $n\ge 0$ and a decreasing sequence of stopping times, 
\beq
\label{tau n}
\tau_n=\inf\{ t\ge 0, \ |Z^1_t-1|= \epsilon_n \}.
\eeq
We hereby outline our argument as follows:
\begin{enumerate}
	\item We first introduce a sequence of geometrically shrinking ``boxes", all centered at 1, where the size of box $n$ equals $\epsilon_n$, which is half that of its predecessor. 
	\item Note that by $t=\epsilon$, an O-U process starting from $1$ on average will wander a distance at least $O(\epsilon^{0.5})\gg \epsilon$ away from 1. So with high probability, the O-U process has already escaped the largest box by time $\epsilon$.
	\item By scaling invariance of B.M., i.e. for any $a>0$, $\frac{1}{\sqrt a} B_{at} \overset{d}{=} B_t$, we can prove that for all $n\ge 0$ with at least a uniformly positive probability $p>0$, an O-U process stays at the right of 1 for some positive fraction of time between $\tau_n$ and $\tau_{n-1}$ to trigger the Poisson jump under the intensity \eqref{continuerate}, and we say the O-U process ``succeeds" in the $n$th step when such an event happens.
	\item We can choose an appropriate constant $\gamma$ independent of $\epsilon$ whose exact value can be found in \eqref{poly gamma}, and $n_1=O(\log \epsilon^{-1})$ s.t. $(1-p)^{n_1}<\frac{\epsilon}{3}$. Then when $\delta< \epsilon^{4\gamma}$ with high probability, any success in step $n\le n_0$ can trigger our Poisson jump.
	\item In order not to trigger the Poisson jump, the O-U process must not jump in all of the first $n_1$ steps. Thus by the strong Markov property, we know that the probability of not jumping is no larger than the product of these uniform upper bounds $(1-p)^{n_1}<\epsilon/3$. 
\end{enumerate}

Now returning to the detailed proof, we firstly show that with high probability $\tau_0 \le \epsilon$.
\begin{lemma}
	For $\tau_0$ in \eqref{tau n}, $\exists \alpha>0$ s.t. $\prob^1(\tau_0>\epsilon) \le \exp(-\epsilon^\alpha)$ for all sufficiently small $\epsilon >0$.
\end{lemma}
\begin{proof}
	Recall that 
	$$
	Z^1_t=1-\int_0^tZ^1_sds +B_t.
	$$
	Let $\bar{\tau}_\epsilon = \inf \{ t>0, \ |B_t|=\epsilon^{\frac{2}{3}} \}$. Noting that $\{\tau_0>\epsilon\} \subset \{\bar{\tau}_\epsilon > \epsilon\} $ and by the scaling invariance and the Markov property of B.M. $B_t$, 
	\bae
	\prob^1(\tau_0>\epsilon) \le& \prob^0(\bar{\tau}_\epsilon > \epsilon)\\ 
	=& \prob^0(\bar{\tau}_1 > \epsilon^{-\frac13}) \\
	\le& (1- \min_{z\in[-1,1]} \prob^z(|B_1|>1))^{ \lfloor  \epsilon^{-\frac13}\rfloor} \\ 
	\le& (1- \prob^0(|B_1|>2))^{ \lfloor  \epsilon^{-\frac13}\rfloor} \\
	\le& \exp \left( -\epsilon^{-\frac14} \right)
	\eae
	for all sufficiently small $\epsilon$, where  $\lfloor \cdot \rfloor$ denotes the integer part. Thus let $\alpha=\frac14$ and the proof is complete.
\end{proof}
Now for any integer $n\ge 1$, we say step $n$ is a ``success" if the event 
\beq
A_n := \left\{\int_{\tau_n}^{\tau_{n-1}} \mathbbm{1}_{ \{Z^1_t>1+\epsilon_n\} }dt>\epsilon_n^2 \right\}
\eeq
happens. To find a lower bound for the probabilities of success in each step, we first consider the following technical lemma for B.M. at scale of order $1$:
\begin{lemma}
	For $B_t$ the standard B.M., let 
	\beq
	\Gamma_x = \inf\{ t>0:\ B_t=x \}.
	\eeq
	Then 
	\beq
	\label{key bound}
	\prob\left( \Gamma_4<\Gamma_{-\frac12}, \Gamma_4<4, \int_0^{\Gamma_{\frac23}} \mathbbm{1}_{ \{B_t \in [\frac14, \frac12]\} } dt >1, 
	\int_0^{\Gamma_{2.8}} \mathbbm{1}_{ \{B_t \in [2.2, 2.5]\} } dt >1 \right) =: p >0.
	\eeq
\end{lemma}
\begin{proof}
	Note that a standard B.M. can approximate any continuous function starting from $0$ with a positive probability. (See Theorem $5.4$ on Page $206$ of \cite{durrett2018stochastic} for details.) Thus by easily choosing a continuous function satisfying the condition in the event of \eqref{key bound}, we conclude the result of interest by noting that Brownian motion can approximate such a continuous function with positive probability.
\end{proof}
\begin{remark}
	The seemingly mysterious constants in \eqref{key bound} are purposely chosen to meet the later needs in the proof of Lemma \ref{event control}. Particularly, we need these constants to create a certain level of ``redundancy'' so that after introducing the drift term, our O-U process can still stay within the intervals of interest. 
\end{remark}
Then by the scaling invariance of $B_t$, we immediately have for all $\theta>0$,
\beq
\prob \left( \Gamma_{4\theta} < \Gamma_{-\frac12\theta}, \Gamma_{4\theta} < 4\theta^2, \int_0^{\Gamma_{\frac23\theta}} \mathbbm{1}_{ \{B_t \in [\frac14\theta, \frac12\theta]\} } dt >\theta^2, 
\int_0^{\Gamma_{2.8\theta}} \mathbbm{1}_{ \{B_t \in [2.2\theta, 2.5\theta]\} } dt >\theta^2 \right) = p .
\eeq
Now we return to the probability of $A_n$. Let $Z_t^{1+\epsilon_n}$ be the O-U process starting from $1+\epsilon_n$, i.e., $Z_t^{1+\epsilon_n} = 1+\epsilon_n - \int_0^t Z_s^{1+\epsilon_n} ds + B_t $. Then by the strong Markov property for the O-U process,
$$
\prob(A_n) \ge \min \left\{\prob \left( \int_{0}^{\tau_{n-1}} \mathbbm{1}_{ \{Z^{1+\epsilon_n}_t>1+\epsilon_n\} }dt>\epsilon_n^2 \right), \prob \left( \int_0^{\tau_{n-1}} \mathbbm{1}_{ \{Z^{1-\epsilon_n}_t>1+\epsilon_n\} }dt>\epsilon_n^2 \right) \right\}
$$
Now letting $\theta=\epsilon_n$, we look at the event 
\beq
E^{\epsilon_n} := \left\{ \Gamma_{4\epsilon_n} < \Gamma_{-\frac12\epsilon_n}, \Gamma_{4\epsilon_n} < 4\epsilon_n^2, \int_0^{\Gamma_{\frac23\epsilon_n}} \mathbbm{1}_{ \{B_t \in [\frac14\epsilon_n, \frac12\epsilon_n]\} } dt >\epsilon_n^2, 
\int_0^{\Gamma_{2.8\epsilon_n}} \mathbbm{1}_{ \{B_t \in [2.2\epsilon_n, 2.5\epsilon_n]\} } dt >\epsilon_n^2 \right\} .
\eeq
\begin{lemma}
	\label{OU control}
	Given the event $E^{\epsilon_n}$, we have $Z_t^{1+\epsilon_n} \in [0,2]$ a.s. $\forall t \in [0, \Gamma_{4\epsilon_n}]$.
\end{lemma}
\begin{proof}
	Otherwise let $\bar{\Gamma}_0$ and $\bar{\Gamma}_2$ be the first time $Z_t^{1+\epsilon_n}$ hits $0$ or $2$ respectively. Without loss of generality, suppose $Z_t^{1+\epsilon_n}$ hits $0$ before $2$ at $[0, \Gamma_{4\epsilon_n}]$ in the event $E^{\epsilon_n}$ and thus we look at the event $\{ \bar{\Gamma}_0< {\Gamma}_{4\epsilon_n}, \bar{\Gamma}_0<\bar{\Gamma}_2 \} \cap E^{\epsilon_n}$. Then within this event there is a.s. $t<4\epsilon_n^2$ s.t. $Z_t^{1+\epsilon_n}\le 0$, $B_t\ge -\frac12\epsilon_n$ and $Z_s^{1+\epsilon_n}<2, \forall s<t$. Note that  
	$$
	Z^{1+\epsilon_n}_t=1+\epsilon_n - \int_0^tZ_s^{1+\epsilon_n}ds +B_t.
	$$
	However,
	$$
	\text{RHS}\ge (1+\epsilon_n) - 2 \cdot 4\epsilon_n^2 -\frac12 \epsilon_n>0 
	$$
	which implies that 
	$$
	\prob\left( \{ \bar{\Gamma}_0< \Gamma_{4\epsilon_n}, \bar{\Gamma}_0<\bar{\Gamma}_2 \} \cap E^{\epsilon_n} \right)=0.
	$$
	Similarly, one also has 
	$$
	\prob\left(\{ \bar{\Gamma}_2< \Gamma_{4\epsilon_n}, \bar{\Gamma}_2<\bar{\Gamma}_0 \} \cap E^{\epsilon_n}\right)=0.
	$$
\end{proof}

Now we are able to show that with at least a uniformly positive probability $p>0$, an O-U process stays to the right of 1 for some positive fraction of time between $\tau_n$ and $\tau_{n-1}$ to trigger the Poisson jump; the detailed proof can be found in Appendix \ref{detail}.
\begin{lemma}
	\label{event control}
	For all sufficiently small $\epsilon>0$ and $n\ge 1$, 
	\bae
	&\prob \left( \int_0^{\tau_{n-1}} \mathbbm{1}_{ \{Z^{1+\epsilon_n}_t>1+\epsilon_n\} }dt>\epsilon_n^2 \right) \ge \prob(E^{\epsilon_n} )=p,\\ 
	&\prob \left( \int_0^{\tau_{n-1}} \mathbbm{1}_{ \{Z^{1-\epsilon_n}_t>1+\epsilon_n\} }dt>\epsilon_n^2 \right)  \ge \prob(E^{\epsilon_n} )=p.
	\eae
Thus we have 
$$
\prob(A_n) \ge \prob(E^{\epsilon_n})=p>0.
$$
\end{lemma}

With the above preparation, we can finish the proof of \eqref{speed aim}.
\begin{proof}
	[Proof of Proposition \ref{polynomial decay}:] For any $\epsilon>0$ (without loss of generality $\epsilon<\frac13$), define $n_1=\left\lfloor \frac{\log(\epsilon^{-1})+\log3}{\log\left(\frac{1}{1-p}\right)} \right\rfloor + 1$ and a constant $\gamma \in (1, +\infty)$ that depends only on $p$ as
	\beq
	\label{poly gamma}
	\gamma:= \frac{4\log2}{\log\left(\frac{1}{1-p}\right)}  + 1
	\eeq
	so we have $\epsilon_{n_1}=\epsilon\cdot 2^{-n_1} \ge \epsilon^\gamma$. Recall that $\Gamma$ obeys the exponential distribution $\exp(1)$ and we define a ``globally failed event",
	$$
	B=\left\{ \Gamma \ge \epsilon^{-\gamma} \right\} \cup \{ \tau_0>\epsilon \} \cup \cap_{n=1}^{n_1} A_n^c
	$$
	and call $B^c$ the globally successful event. By Lemmas \ref{OU control} and \ref{event control} we have 
	\bae
	\prob(B)\le& \exp(-\epsilon^{-\gamma}) + \exp( -\epsilon^{-\frac14} ) + (1-p)^{n_1}\\
	<& \frac13\epsilon + \frac13\epsilon + \frac13\epsilon <\epsilon.
	\eae
	Recalling the random set $\hat{I}(\eta_1,\delta)$ in \eqref{calculations}, then for any $\delta< \epsilon^{4\gamma} <\epsilon_{n_1}$, within $B^c$ we have $\{\tau_0\le\epsilon\}$ and $\exists$ $k\in[1,n_1]$ s.t. 
	\bae
	\frac{\mathfrak{L}\left(\hat{I}(\epsilon, \delta)\right)}{\delta} = \frac{\int_0^{\epsilon} \mathbbm{1}_{ \{Z^1_t > 1+\delta \} }dt}{\delta} \ge \frac{\int_{\tau_k}^{\tau_{k-1}} \mathbbm{1}_{ \{Z^1_t > 1+\epsilon_k\} }dt}{\delta} \ge \frac{\epsilon^2_{n_1}}{\delta} \ge \epsilon^{-2\gamma}>\Gamma,
	\eae
	which gives \eqref{speed aim}. Thus we have found a polynomial order of $\eta_1$ as the lower bound for $\eta_0$ and the same for $\eta_2$ and the proof of \eqref{conv rate} is complete.	
\end{proof}


\section{Numerical tests}\label{numerical}

In previous sections we have shown that the state-dependent jump-diffusion process \(X_t^\delta\) converges to \(X_t\) in distribution with a polynomial-order convergence rate. However, quantifying the correct convergence rate of \(X_t^\delta\) remains an open question. Recently, a structure-preserving numerical scheme for the Fokker-Planck equation \eqref{eq_planck} based on Scharfetter-Gummel reformulation was proposed in \cite{Yantong}. With this numerical scheme, we aim to explore the convergence structure and study \(X_t^\delta\) in terms of density functions through the Fokker-Planck equation \eqref{eq:FKJD} together with its nonlinear cases. Numerical study of the density function \(f^\delta\) not only provides numerical evidence of the convergence rate of the process, but also indicates signs of self-similar structure when \(\delta\to0\).

This section is outlined as follows. First, in Section~\ref{sec:intro}, we introduce the  Scharfetter-Gummel reformulation on Fokker-Planck equation and the detailed construction of the numerical schemes. Then, in Section~\ref{sec:conv}, we numerically examine the convergence rate of the approximation error. Last, in Section~\ref{sec:self}, we study the self-similar structure of \(f^\delta\) when it vanishes on \((1,+\infty)\) as \(\delta\to0\). We also note that the Fokker-Planck equation \eqref{f0delta equation} for the killing process \(\widetilde{X_t^\delta}\) (defined in Equation \eqref{OUkilling}) is also considered in this section.

\subsection{Introduction to the scheme}
\label{sec:intro}

First we introduce the nonlinear extensions of the Fokker-Planck equations for the jump-diffusion process  \(X_t^\delta\) and the killing process \(\widetilde{X_t^\delta}\), which are similar to the Fokker-Planck equation \eqref{master equation0} associated with \(X_t\). We only show the following nonlinear equation for the density function $f^\delta$ of the jump-diffusion process \(X_t^\delta\);  the nonlinear equation for the density function \(f_0^\delta\) of the killing process \(\widetilde{X_t^\delta}\) can be derived in a similar way.
\begin{equation}
	\label{EqNonlinearD}
	\left\{
	\begin{aligned}
		&\frac{\partial f^{\delta}}{\partial t}(x,t)+\frac{\partial}{\partial x}[(-v+bN^{\delta}(t)) f^{\delta}(x,t)]- a(N^\delta(t))\frac{\partial^2 f^{\delta}}{\partial x^2} (x,t) \\ &=N^\delta(t)\delta(x)-\lambda^{\delta}(x) f^{\delta}(x, t), \quad (x,t) \in (- \infty,1)\times [0,+\infty), \\
		&N^\delta (t) =\int_{\mathbb R} \lambda^\delta(y)f^\delta(y,t)dy,
		f^\delta(-\infty,t)=f^\delta(+\infty,t)=0, \quad t>0,
	\end{aligned}
	\right. 
\end{equation}
where the terms \(a(N^\delta(t))\) and \(-v+bN^{\delta}(t)\) incorporate the effect of the mean firing rate on the dynamics of the density function at the macroscopic level. In particular, \(b\) models the connectivity of the neuron networks: \(b>0\) describes excitatory networks and \(b<0\) describes inhibitory networks. In this section, we assume \(a\equiv1\) and we are concerned with the convergence behavior with different connectivity parameters \(b\).

The Scharfetter-Gummel reformulation on Equation \eqref{EqNonlinearD} is given as follows:
\begin{eqnarray}\label{SGreform}
	\frac{\partial f^{\delta}}{\partial t}(x,t)-a\frac{\partial}{\partial x}\left(M^\delta(x,t)\partial_x\left(
	\frac{f^\delta(x,t)}{M^\delta(x,t)}\right)\right) =N^\delta(t)\delta(x)-\lambda^{\delta}(x) f^{\delta}(x, t),\nonumber\\ (x,t)\in(- \infty, 1)\times [0,+\infty),
\end{eqnarray}
where 
\[
M^\delta(x,t)=\exp \left(-\frac{(x-b N^\delta(t))^{2}}{2 a}\right).
\]
The numerical scheme for Equation \eqref{EqNonlinearD} is based on this reformulated equation.

Even though the jump-diffusion process \(X_t^\delta\) is of better regularity than \(X_t\), numerical approximation of \(f^\delta(x,t)\) near \(x=1\) is still at risk of being inaccurate especially when $\delta$ is close to 0. Therefore, we apply the logistic scaling of the density function to partition a denser grid around \(x=1\). We take the computation domain as \([-4,4]\) and assume homogeneous Dirichlet boundary condition for the density functions. We make the substitution
\begin{equation}\label{eq:subs}
	y=h_L(x)=\frac{1}{1+e^{-\left(x-1\right)}},\quad x\in[-4,4],
\end{equation}
and denote \(q^\delta(y,t)=f^\delta(g_L(y),t)\), where \(g_L\) stands for inverse function of the logistic function \(h_L\). Figure \ref{fig:L} shows an illustration of the scaling.

\begin{figure}
	\centering
	\includegraphics[width=10cm,height=8cm]{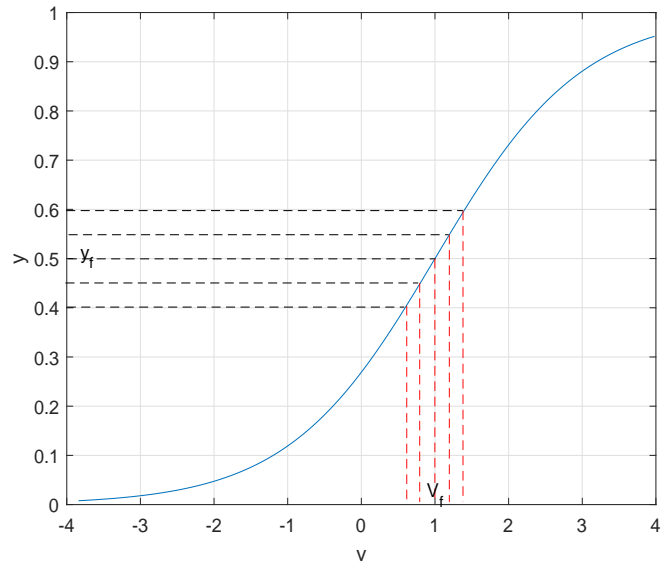}
	\caption{Illustration of the logistic scaling. In the figure, we partition a denser grid near \(V_f=1\) through the logistic scaling.}
	\label{fig:L}
\end{figure}

Then we put Equation \eqref{eq:subs} into Equation \eqref{SGreform} to derive an equation for \(q^\delta(y,t)\) on the computational domain
\begin{eqnarray}\label{eq:subsFKmod}
	\frac{\partial q^\delta}{\partial_{t}}(y, t)- \frac{a}{g_L'(y)}\partial_{y}\left(\frac{M^\delta\left(g_L(y),t\right)}{g_L'(y)} \partial_{y}\left(\frac{q^\delta(y, t)}{M^\delta\left(g_L(y),t\right)}\right)\right)=N^\delta(t)\delta\left(y-y_r\right)\nonumber\\-\lambda^\delta\left(g_L(y)\right)q^\delta(y,t),\quad(y,t)\in\left[\frac{1}{1+e^5},\frac{1}{1+e^{-3}}\right]\times[0,+\infty),
\end{eqnarray}
where \(g_L'(y)=\frac{1}{y-y^2}\) and the reset point \(y_r=h_L(0)=\frac{1}{1+e}\). In addition, the mean firing rate function \(N^\delta\) is given as follows:
\begin{equation}\label{firingrateDeter}
	N^{\delta}(t) =\int_{\frac{1}{1+e^5}}^{\frac{1}{1+e^{-3}}} g_L'(z)\lambda^{\delta}\left(g_L(z)\right) q^{\delta}(z, t) d z.
\end{equation}

The numerical scheme applied in this section is based on discretization of Equations \eqref{eq:subsFKmod} and \eqref{firingrateDeter}. Let \(q^\delta_{j,m}\) stand for the numerical value of \(q^\delta\) at \(y_j=jh+\frac{1}{1+e^5}\in[\frac{1}{1+e^5},\frac{1}{1+e^{-3}}]\) and \(t_m=m\tau\geq0\), where \(h\) and \(\tau\) denote spatial and temporal step lengths. Let \(N_m^\delta\) denote the numerical approximation of the firing rate function \(N^\delta\) at \(t_m\). Also note that the reset point \(y_r\) is a grid point denoted as \(y_r=y_D\) where \(D\in\mathbb{N}^*\).

We apply the semi-implicit scheme to discretize the equations \cite{Yantong}. In other words, we treat \(q^\delta_{j}\) implicitly but treat \(N^\delta\) explicitly (including the \(N^\delta\) in the term \(M^\delta(x,t)\)). The scheme is as follows:
\begin{eqnarray}
	&&\frac{q^\delta_{j,m+1}-q^\delta_{j,m}}{\tau}-\frac{a}{h^2g_L'\left(y_j\right)}\left(\frac{M^\delta\left(g_L\left(y_{j+\frac{1}{2}}\right),t_m\right)}{g_L'\left(y_{j+\frac{1}{2}}\right)}\left(\frac{q^\delta_{j+1,m+1}}{M^\delta\left(g_L\left(y_{j+1}\right),t_m\right)}-\frac{q^\delta_{j,m+1}}{M^\delta\left(g_L\left(y_j\right),t_m\right)}\right)\right.\nonumber\\
	&&\left.-\frac{M^\delta\left(g_L\left(y_{j-\frac{1}{2}}\right),t_m\right)}{g_L'\left(y_{j-\frac{1}{2}}\right)}\left(\frac{q^\delta_{j,m+1}}{M^\delta\left(g_L\left(y_{j}\right),t_m\right)}-\frac{q^\delta_{j-1,m+1}}{M^\delta\left(g_L\left(y_{j-1}\right),t_m\right)}\right)\right)\nonumber\\
	&&=\frac{1}{h}N_n^\delta I(y_j=y_D)-\lambda^\delta\left(g_L\left(y_j\right)\right)q^\delta_{j,m+1},
\end{eqnarray}
and
\begin{equation}
	N^\delta_m=h\sum_{j}g_L'\left(y_j\right)\lambda^\delta\left(g_L\left(y_j\right)\right)q_{j,m}^\delta,
\end{equation}
where \(I(y)\) is the indicator function.

In our numerical simulations, we simulate excitatory and inhibitory networks with different connectivity parameters \(b\). We use the same Gaussian distribution as the initial condition for all of the numerical tests:
\begin{equation}\label{eq:Max}
	f_{G}(x)=\frac{1}{\sqrt{2 \pi} \sigma_{0}} e^{-\frac{\left(x-x_{0}\right)^{2}}{2 \sigma_{0}^{2}}},
\end{equation}
where \(x_0=-1\) and \(\sigma_0^2=0.01\) are two given parameters. Since we assume that \(X_t^\delta\) and \(\widetilde{X_t^\delta}\) start at a given point in previous sections, we choose \(\sigma_0\) to be very small in order to approximate the one-point initial distribution of the processes. The computing time is fixed to \(t_{\max}=1\). In addition, we consider the rate function \(\lambda^\delta(x)\) as follows:
\begin{equation}\label{eq:disrate}
	\lambda^{\delta}(x)=\left\{\begin{array}{ll}
		0, & x \leq 1, \\
		\frac{1}{\delta}, & x \geq 1,
	\end{array}\right.
\end{equation}
which is slightly different from the continuous rate of Equation \eqref{continuerate}.

\begin{remark}
	
	In fact, with the same initial data, numerical solutions of the Fokker-Planck equations with two different rate functions are almost the same. Convergence and asymptotic behavior for the two cases (Equations \eqref{continuerate} and \eqref{eq:disrate}) are similar, though the convergence exponents are slightly different. However, \eqref{eq:disrate} is of a simpler form, which facilitates the convergence study. Hence, we only consider the rate function defined in Equation \eqref{eq:disrate} in numerical tests.
	
\end{remark}

\subsection{Convergence rates}
\label{sec:conv}

In this subsection we aim to investigate the convergence of the jump-diffusion process \(X_t^\delta\) (defined in Equation \eqref{eq:OU2}) and the killing process \(\widetilde{X_t^\delta}\) (defined in Equation \eqref{OUkilling}) as \(\delta\to0\) through numerical examination of the Fokker-Planck equations of the two processes (see Equations \eqref{eq:FKJD} and \eqref{f0delta equation}). We compute the discrepancies between the density functions and firing rate functions (defined in Equation \eqref{EqNonlinearD}) of the two processes, i.e. we consider  density discrepancy 
\begin{equation}\label{disc1}
	D^f(\delta)=\left\|f^\delta(x,1)-f(x,1)\right\|_{\infty},\quad D^{f_0}(\delta)=\left\|f_0^\delta(x,1)-f_0(x,1)\right\|_{\infty},
\end{equation}
and firing rate discrepancy 
\begin{equation}\label{disc2}
	D^N(\delta)=\left\|N^\delta(t)-N(t)\right\|_\infty,\quad D^{N_0}(\delta)=\left\|N_0^\delta(t)-N_0(t)\right\|_\infty,
\end{equation}
where \(\|\cdot\|_{\infty}\) denotes the \(L^\infty\) norm in space or time. Here the density functions \(f(x,t),f_0(x,t)\) and firing rate functions \(N(t),N_0(t)\) of \(X_t\) and its killing process \(\widetilde{X_t}\) (defined in a similar way using Equation \eqref{OUkilling}) are obtained by numerically solving the nonlinear Fokker-Planck equation \eqref{master equation0} with nonlinear drift and diffusion term \eqref{nonlinear}  using the scheme in Section~\ref{sec:intro}.

In Figure \ref{fig:rate}, we show the results of simulating the cases \(\delta=\frac{1}{2^k}\) (\(0\leq k\leq 7,k\in\mathbb{N}\)) with different parameters \(b\), where we consider the evolution of the discrepancy functions defined in Equations \eqref{disc1} and \eqref{disc2} as \(\delta\) goes to 0. The convergence of the density functions and firing rate functions are roughly linear when \(\delta\) is moderately small, while the rates of convergence of the cases with different connectivity parameters \(b\) vary.

Then we define the convergence rates for the discrepancy functions as \(\delta\to0\) as follows:
\begin{equation}\label{defRate}
	D^f(\delta)=A_f\delta^{R^f}+o(\delta^{R^f}),\delta\to0,
\end{equation}
where \(R^f\) denotes the convergence rate of density function \(f^\delta\) and \(A_f\) is a fixed parameter. We can define convergence rates \(R^{f_0}, R^N, R^{N_0}\) for the discrepancy functions in Equations \eqref{disc1} and \eqref{disc2} in the similar way to Equation \eqref{defRate}.

In Table \ref{tab:rate1} we show the convergence rates of the functions with different \(b\). The convergence rates are computed through linear fitting after eliminating the data when \(\delta\) is too large or too small in order to avoid inaccuracy. More specifically, we retain only the data when \(\delta=\frac{1}{2^k}\) where \(k=4,5,6,7\) for the linear fitting.

We remark that the connectivity parameter $b$ is chosen to be moderately small such that the solution of \eqref{master equation0}  with non-linear terms \eqref{nonlinear} does not blow up, and we have observed that $f^\delta(x,t)$ converges to $f(x,t)$ as $\delta \rightarrow 0$. However, when $b>0$ increases, the solution to \eqref{master equation0}  with non-linearity \eqref{nonlinear} may blow up in finite time while the solution of \eqref{EqNonlinearD} remains globally well posed. In fact, time periodic solutions have been shown to exist or been numerically observed for variant regularized models. The interested reader may refer to \cite{cormier2020hopf} and \cite{ikeda2021theoretical} for detailed discussions.

\subsection{Self-similar structure}
\label{sec:self}

As \(\delta\to0\), the jump-diffusion process \(X_t^\delta\) converges to \(X_t\), which takes values on the half space \((-\infty,1]\) rather than the whole space. The exact process of how \(X_t^\delta\) vanishes on \((1,+\infty)\) remains an open question. In this subsection, we aim to study  the self-similar profile of the density function \(f^\delta\) of \(X_t^\delta\) on \((1,+\infty)\) through numerical experiments. 

\begin{center}
	
	\begin{tabular}{c|ccccc}
		\toprule
		Jump-diffusion process \(X^\delta_t\) & \(b=1\) & \(b=0.5\) & \(b=0\) & \(b=-0.5\) & \(b=-1\)  \\
		\midrule
		Convergence rate \(R^f\)& 0.3716& 0.3187&	0.3505& 0.3766& 0.3961 \\
		Convergence rate \(R^{f_0}\) &0.3365& 0.3832& 0.4092&	0.4262& 0.4307 \\
		\midrule
		Killing process \(\widetilde{X_t^\delta}\)  & \(b=1\) & \(b=0.5\) & \(b=0\) & \(b=-0.5\) & \(b=-1\)  \\ \midrule
		Convergence rate    \(R^{N}\) &0.3416& 0.3856&	0.4122& 0.4219& 0.4228  \\
		Convergence rate \(R^{N_0}\) & 0.4166& 0.4302& 0.4309& 0.4322& 0.4302 \\
		\bottomrule
	\end{tabular}
	\captionof{table}{Convergence rates for density functions and firing rate functions of the Fokker-Planck equations of two processes with different connectivity parameter \(b\). See Equation \eqref{defRate} for the definition of the convergence rates of the functions. See Equations \eqref{eq:FKJD}, \eqref{f0delta equation} and \eqref{EqNonlinearD} for the Fokker-Planck equations of the jump-diffusion process \(X^\delta_t\) and the killing process \(\widetilde{X_t^\delta}\).}
	\label{tab:rate1}
\end{center}

We assume an ansatz for \(f^\delta\) when \(\delta\to0\) as follows:
\begin{equation}\label{eq:fdeltaEstimate}
	f^\delta(x,t)=\delta^\alpha\psi\left(\delta^\beta(x-1)\right)+o\left(\delta^\alpha\right),\quad \forall x\in[1,+\infty),
\end{equation}
where \(\psi\) is defined on \(\mathbb{R}^+\) and \(\alpha,\beta\) are two fixed parameters. In this subsection, we aim to explore the self-similar structure of \(f^\delta\) (defined in Equation \eqref{eq:FKJD}) with such an ansatz and find the fixed parameters \(\alpha\) and \(\beta\) numerically. Moreover, we also make a similar ansatz for the density \(f_0^\delta\) of the killing process \(\widetilde{X_t^\delta}\) (defined in Equation \eqref{f0delta equation}) as a reference, since the \(f_0^\delta\) display a similar vanishing structure to \(f^\delta\).

Numerical examinations involve the cases of \(\delta=\frac{1}{2^k}\) (\(0\leq k\leq 7,k\in\mathbb{N}\)) with different parameters \(b\). Through similar data choices and linear fitting, numerical results for \(\alpha\) and \(\beta\) are shown in Table \ref{tab:rate2}.

Finally, in Figure \ref{fig:psi}, we take the values \(\alpha\) and \(\beta\) in Table \ref{tab:rate2} in ansatz \eqref{eq:fdeltaEstimate} for \(X_t^\delta\) and plot the profiles of \(\psi\) for each \(\delta\) with connectivity parameter \(b=0,1\). Numerically, \(\psi\) is nearly independent of \(\delta\) and decays exponentially in \(y\). Therefore, we conclude that it is very likely that \(f^\delta(x)\) exhibits the self-similar profile in Equation \eqref{eq:fdeltaEstimate} when \(x\geq1\).

\begin{center}
	\begin{tabular}{c|ccccc}
		\toprule
		Jump-diffusion process \(X^\delta_t\) & \(b=1\) & \(b=0.5\) & \(b=0\) & \(b=-0.5\) & \(b=-1\)  \\
		\midrule
		Values of \(\alpha\)& 0.2713&0.3187&0.3505&0.3766&0.3961 \\
		Values of \(\beta\) &-0.4256&-0.4363&-0.4283&-0.4317&-0.4268\\ \midrule 
		Killing process \(\widetilde{X_t^\delta}\)  & \(b=1\) & \(b=0.5\) & \(b=0\) & \(b=-0.5\) & \(b=-1\)  \\
		\midrule
		Values of \(\alpha\) &0.3448&0.3856&0.4122&0.4344&0.4507\\
		Values of \(\beta\) &-0.4307&-0.4148&-0.4136&-0.4317&-0.4317\\
		\bottomrule
	\end{tabular}
	\captionof{table}{Numerical values of the parameters \(\alpha\) and \(\beta\) in ansatz \eqref{eq:fdeltaEstimate}. Parameters of the two processes with different connectivity \(b\) are shown in the table. See Equations \eqref{eq:FKJD}, \eqref{f0delta equation} and \eqref{EqNonlinearD} for the Fokker-Planck equations.}
	\label{tab:rate2}
\end{center}

\begin{figure}
	\centering
	\includegraphics[width=0.48\textwidth]{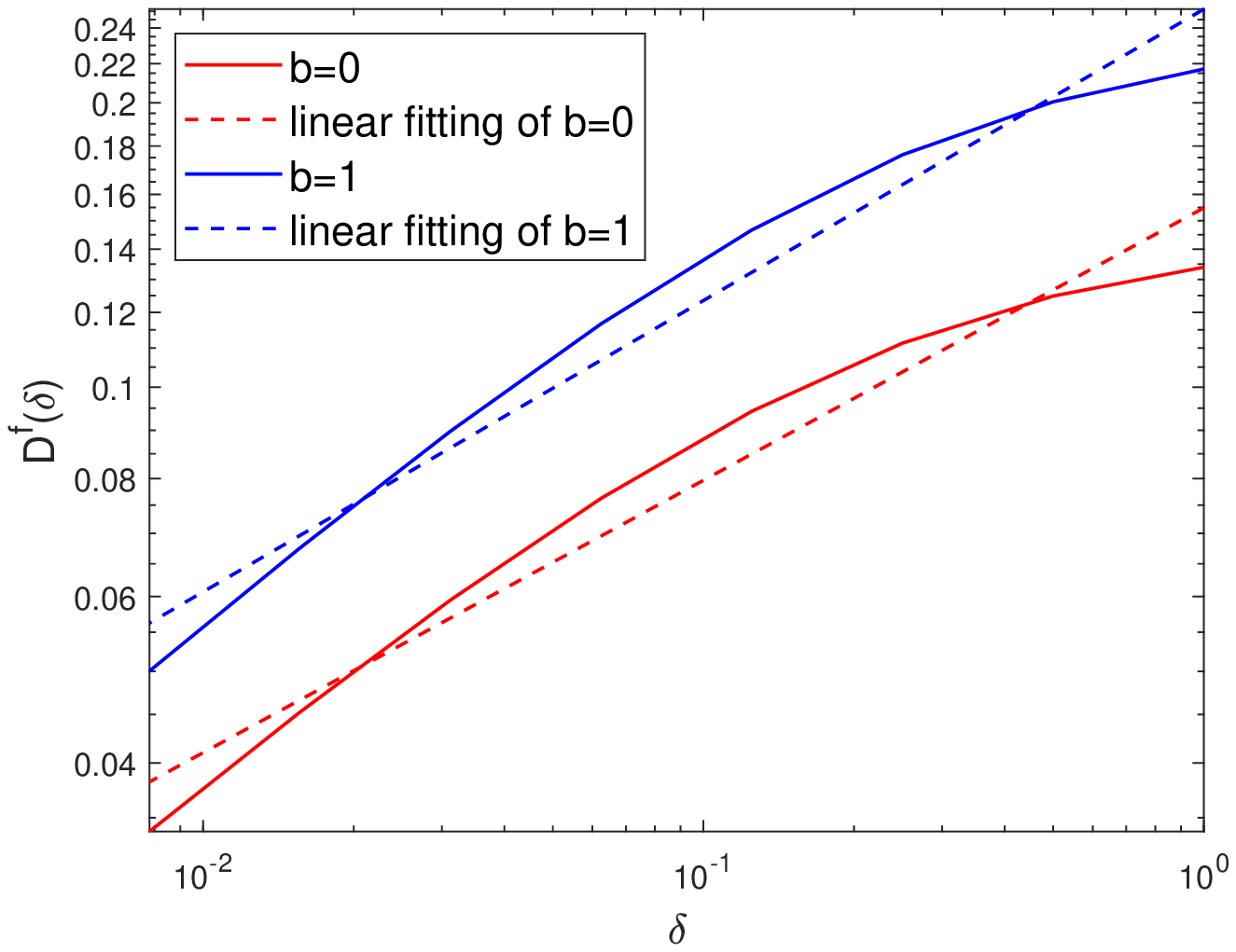}
	\includegraphics[width=0.48\textwidth]{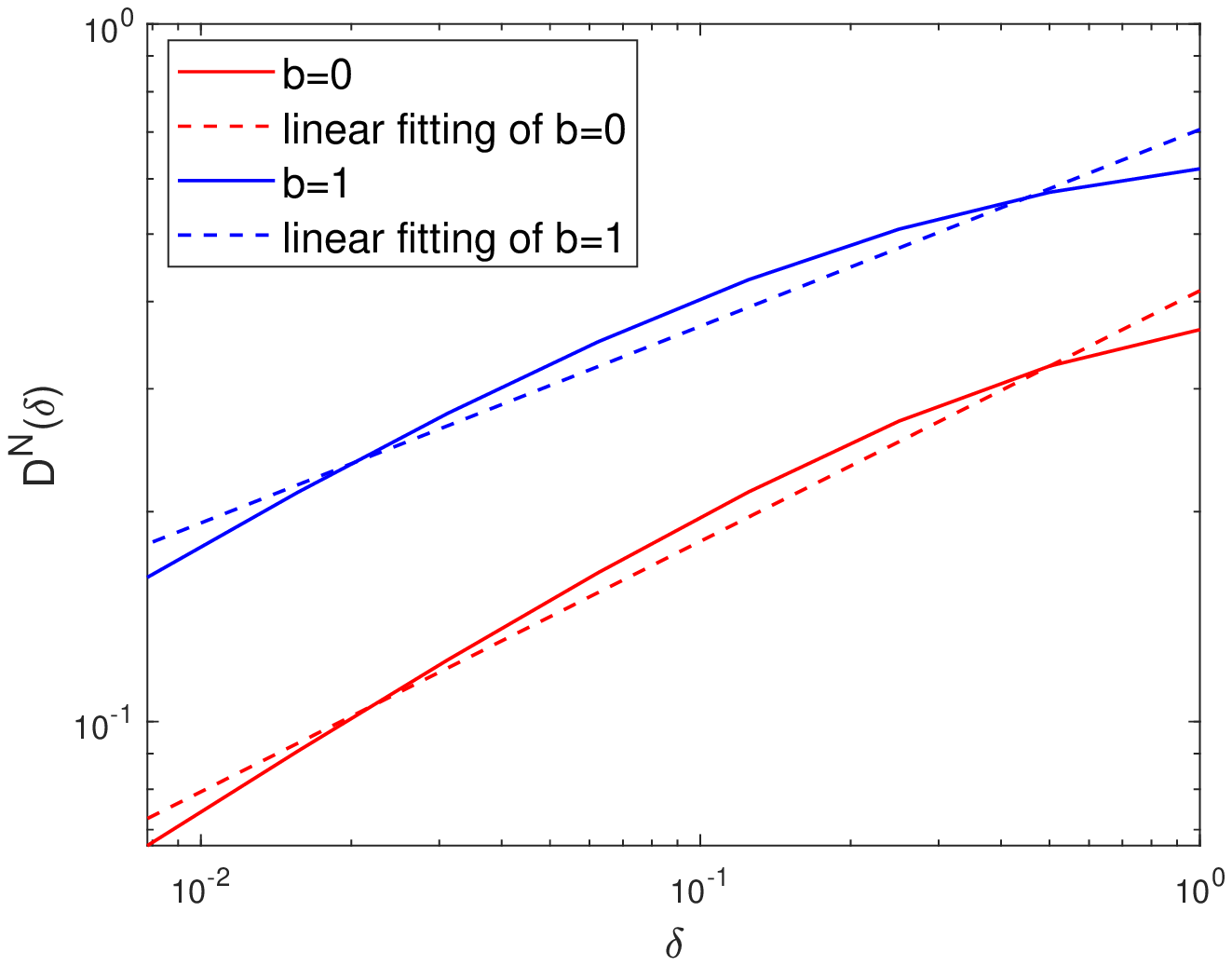}
	\includegraphics[width=0.48\textwidth]{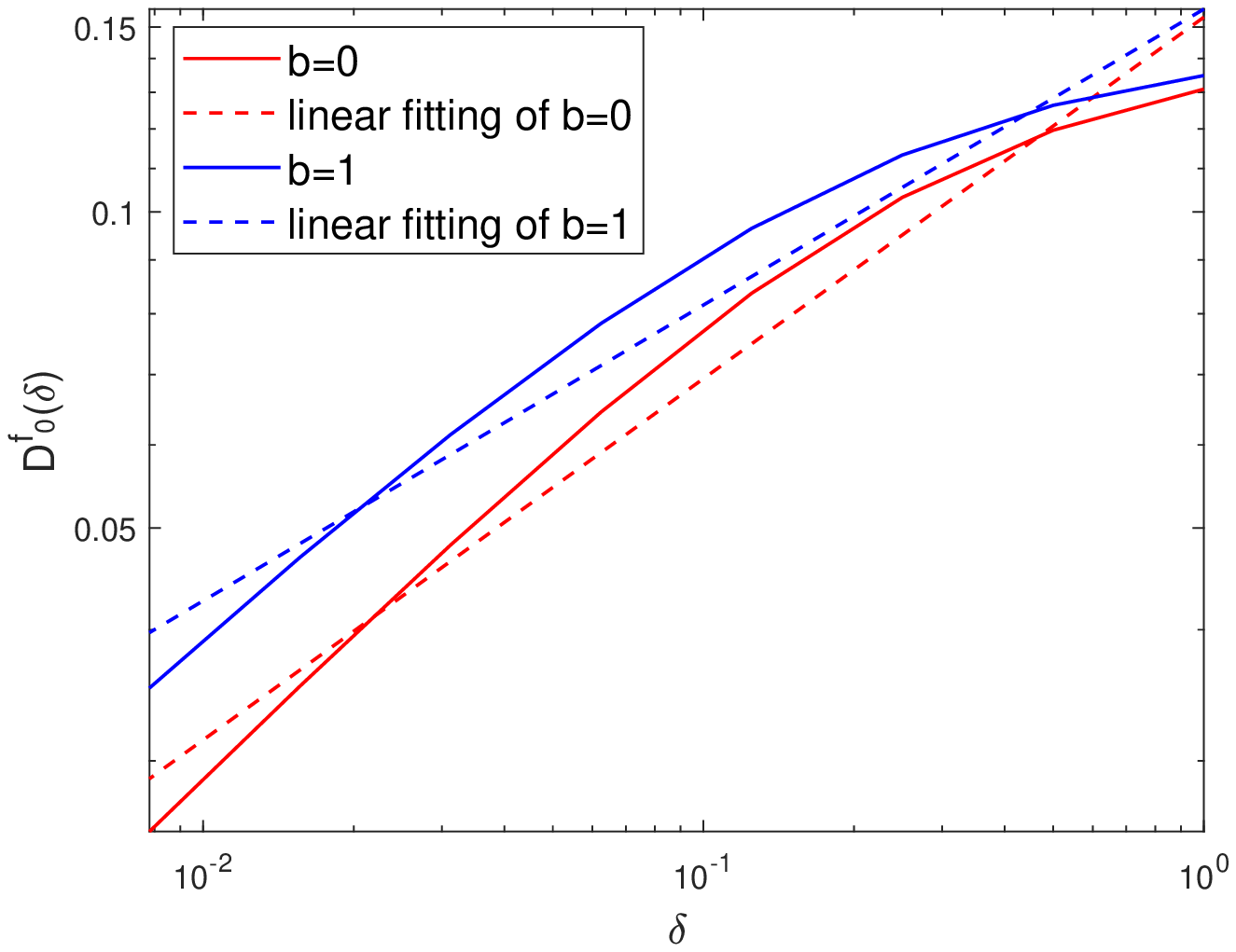}
	\includegraphics[width=0.48\textwidth]{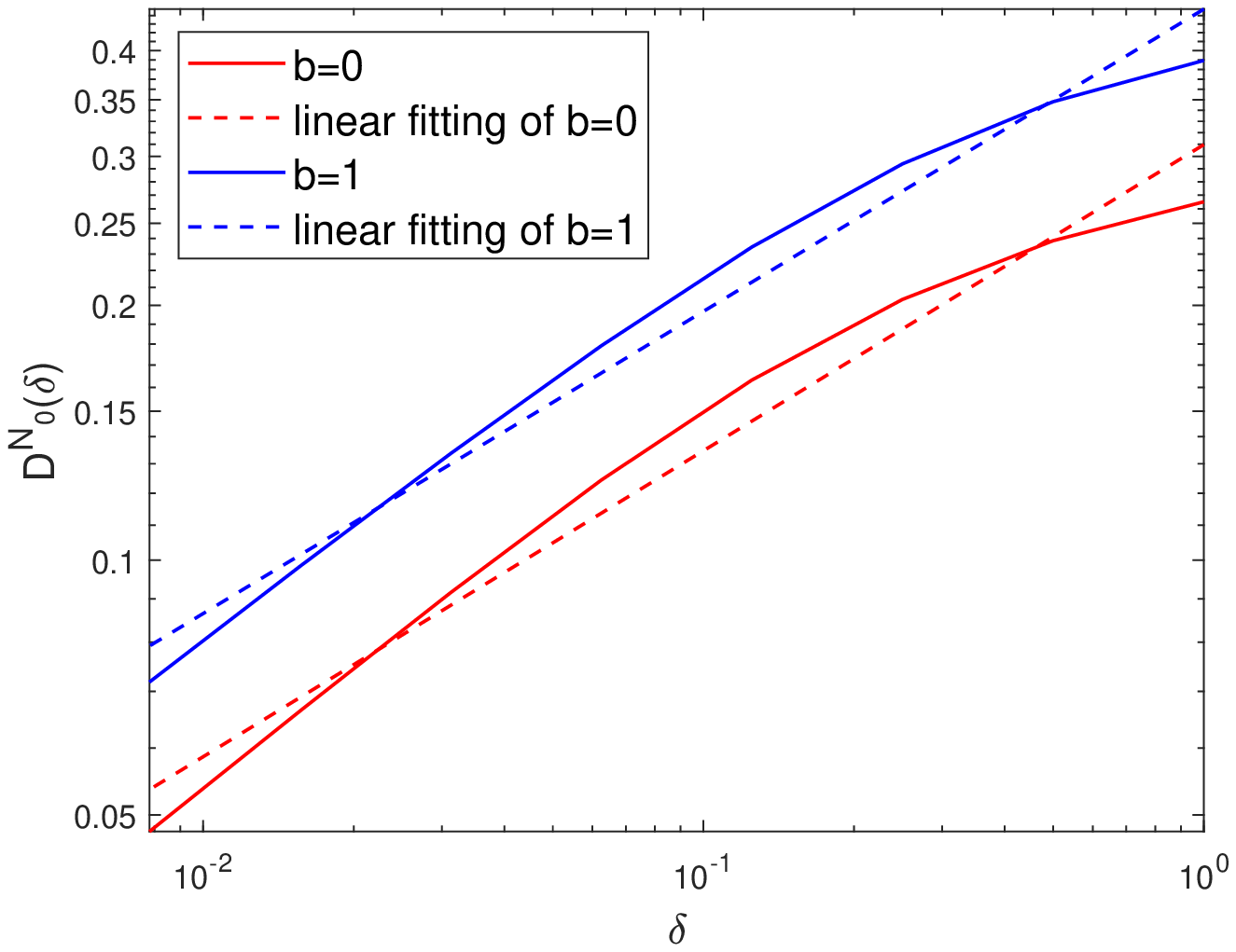}
	\caption{(Convergence of the density functions and the firing rate functions of the jump-diffusion process \(X_t^\delta\) and killing process \(\widetilde{X_t^\delta}\) with parameter \(b=0\) and \(b=1\).) These figures show the evolution of the discrepancy functions defined in Equations \eqref{disc1} and \eqref{disc2} as \(\delta\to0\). Top: the density discrepancy \(D^f(\delta)\) and firing rate discrepancy \(D^N(\delta)\) of the jump-diffusion process \(X_t^\delta\) with connectivity parameter \(b=0\) and \(b=1\). See Equations \eqref{eq:FKJD} and \eqref{EqNonlinearD} for the Fokker-Planck equations of the process \(X_t^\delta\). Bottom: the density discrepancy \(D^{f_0}(\delta)\) and firing rate discrepancy \(D^{N_0}(\delta)\) of the killing process \(\widetilde{X_t^\delta}\) with connectivity parameter \(b=0\) and \(b=1\). See Equations \eqref{f0delta equation} and \eqref{EqNonlinearD} for the Fokker-Planck equations of the process \(\widetilde{X_t^\delta}\)}.
	\label{fig:rate}
\end{figure}

\begin{figure}
	\centering
	\includegraphics[width=0.48\textwidth]{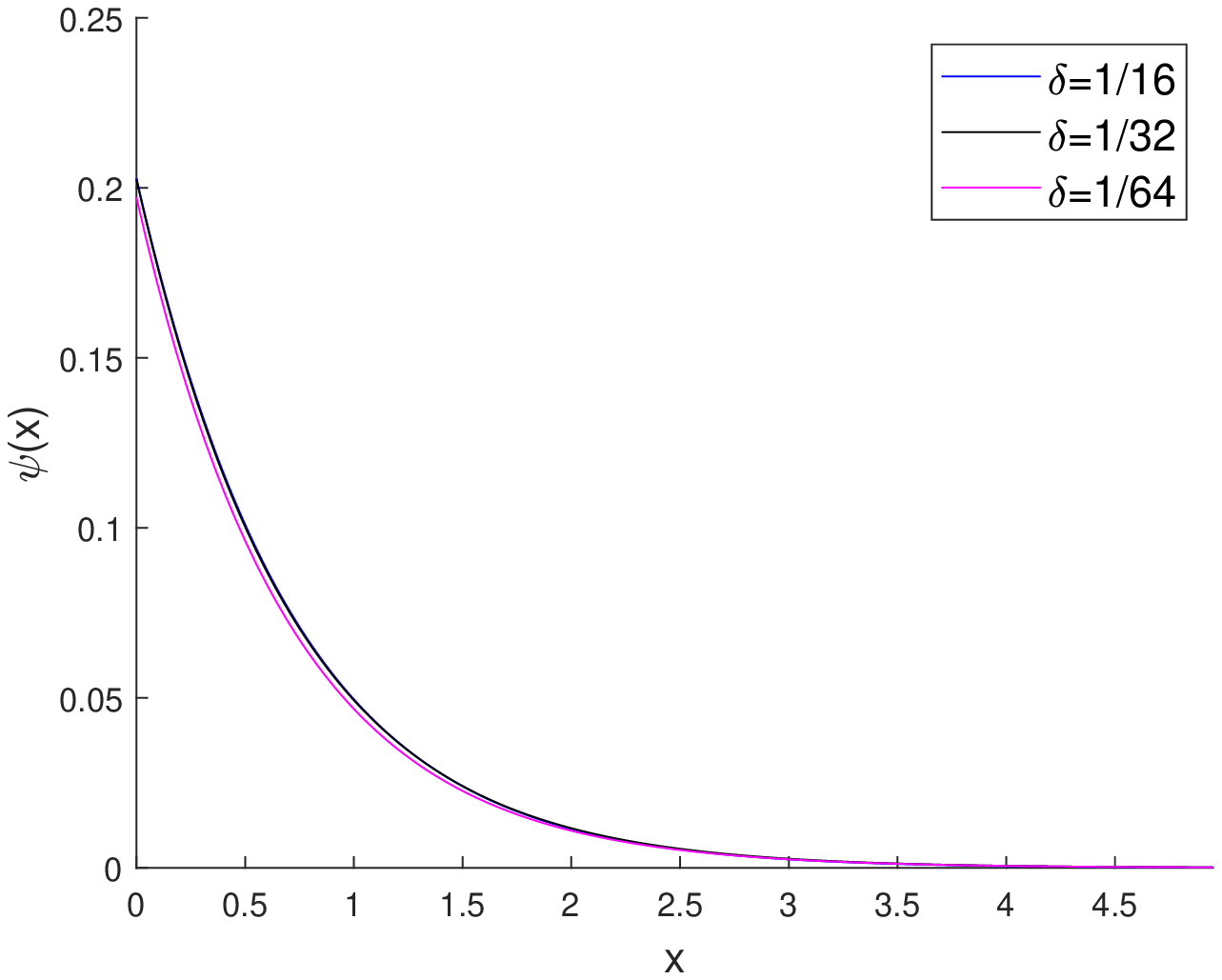}
	\includegraphics[width=0.48\textwidth]{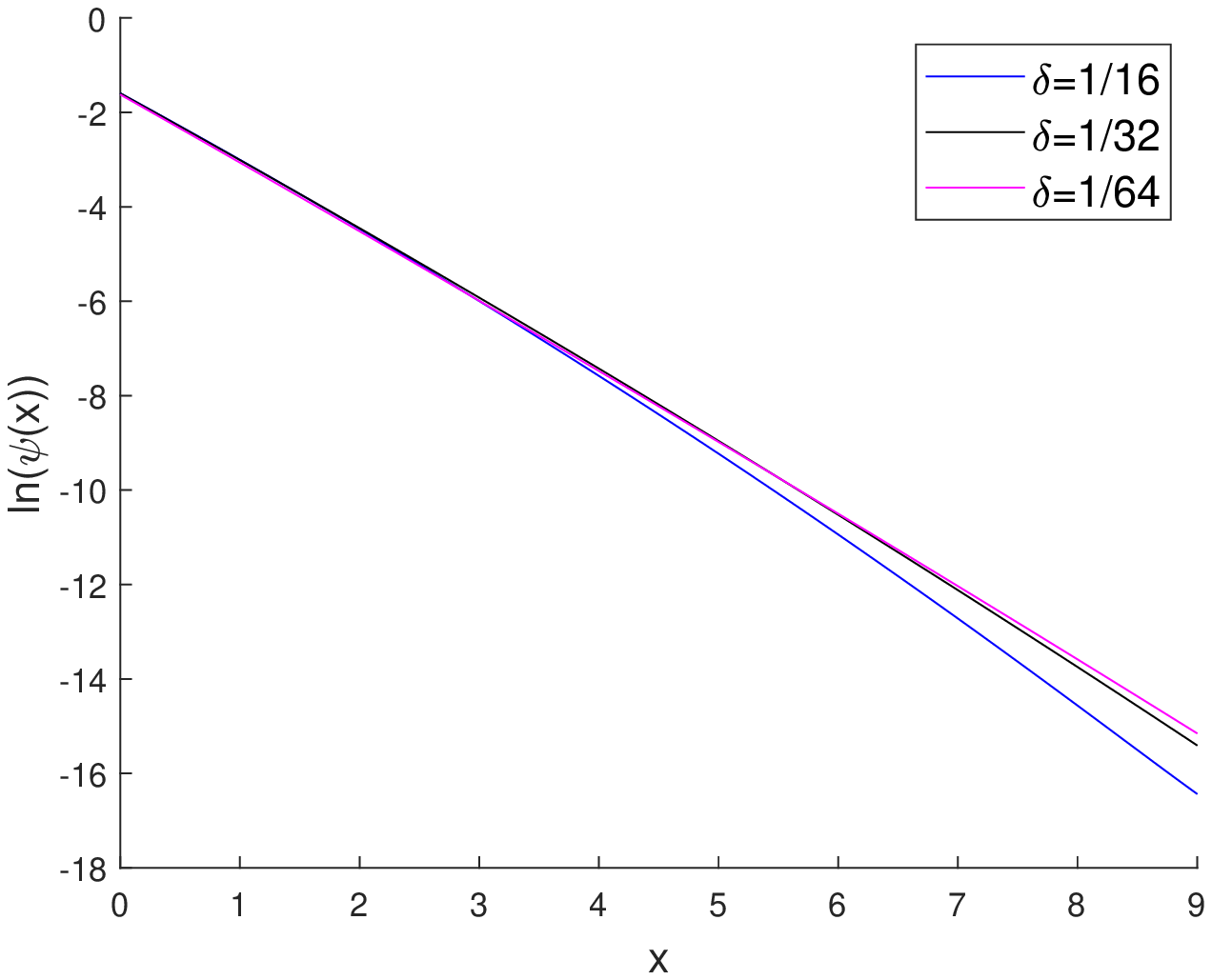}
	\includegraphics[width=0.48\textwidth]{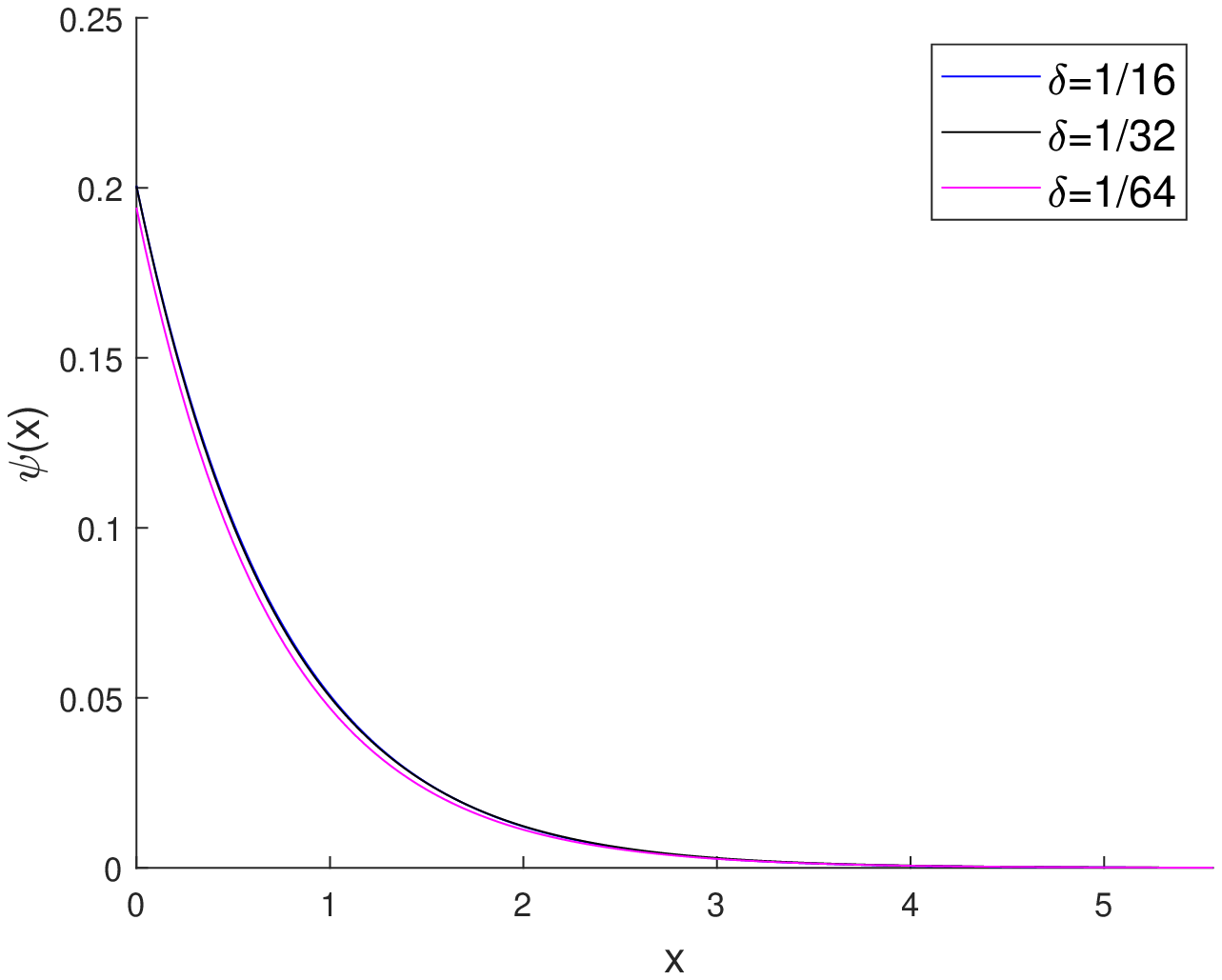}
	\includegraphics[width=0.48\textwidth]{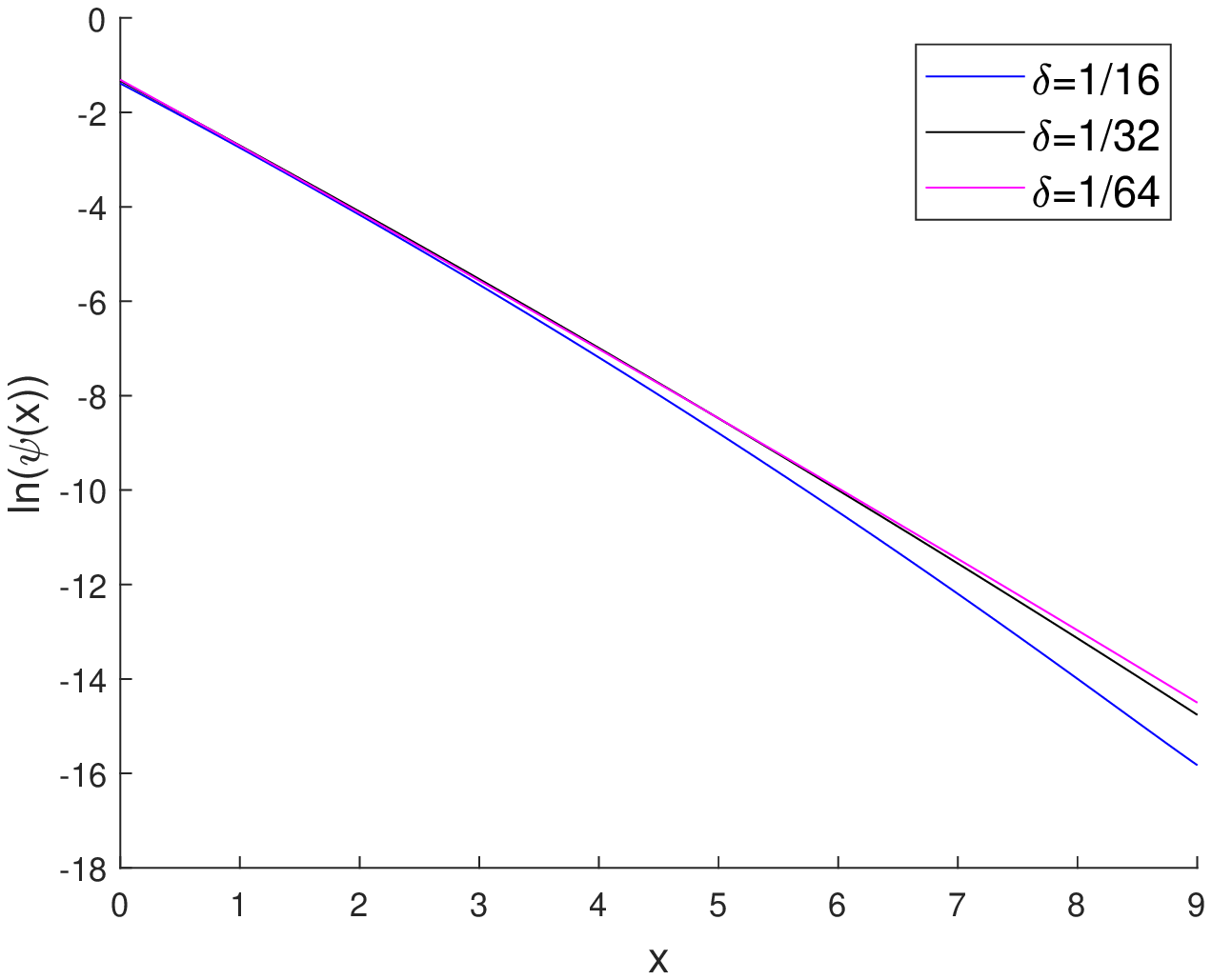}
	\caption{ (Image for \(\psi\) and \(\ln(\psi)\) with connectivity parameter \(b=0,1\).) In this figure, we put the values of \(\alpha\) and \(\beta\) in Table \ref{tab:rate2} into the ansatz \eqref{eq:fdeltaEstimate} 
		for the density function \(f^\delta\) with different \(\delta\). Numerically, we see the profile of \(\psi\) is independent of \(\delta\), which indicates that \(f^\delta\) is likely to exhibit self-similar structure when \(x\geq1\) as \(\delta\to0\). See Equations \eqref{eq:FKJD} and \eqref{EqNonlinearD} for the Fokker-Planck equations of the process \(X_t^\delta\). Left: Image for \(\psi\). Right: Image for \(\ln(\psi)\). Top: \(b=0\). Bottom: \(b=1\).}
	\label{fig:psi}
\end{figure}

\section{Conclusion and Discussion}\label{conclusion}

In this work we aim to reduce the gap in understanding between the mean-field integrate-and-fire model as a stochastic process and the PDE model as an evolving density function. As shown in \cite{caceres2011analysis}, it is possible to find an initial probability distribution such that the solution of the nonlinear Fokker-Planck equation with a deterministic firing potential must blow up in finite time, which is conjectured to be linked to the multiple firing events (synchronization) of neuronal networks. The random discharge mechanism is introduced to prevent the blow up of the solution of the PDE model such that the synchronized state becomes possible on the macroscopic level. In this paper, we have rigorously justified that the regularized solution is indeed an approximation to the original one, which confirms the scientific intuition behind the random discharge mechanism. As the continuation of \cite{liu2020rigorous}, we only focus on the linear cases and show that the relevant random discharge model converges to the original integrate-and-fire model in distribution as the regularization parameter goes to $0$. Mathematically, the iterated scheme can effectively reduce the difficulties of analyzing the problems with the firing-and-resetting mechanism, and gives more intuitive stochastic interpretations of the macroscopic quantities of the PDE, which are otherwise obscure. Using specifically designed numerical experiments, we have observed evidence for the convergence rate and the asymptotic behavior for both the linear cases and the more sophisticated nonlinear cases, which motivates us to carry out a rigorous asymptotic analysis in subsequent work. It is worth noting that we have not yet incorporated the dependence on the mean firing rate in the drift velocity and in the diffusion coefficient and we shall investigate those directions in later work. However, there are still additional challenges due to the interacting nature and the nonlinearity within the model.

\section*{Appendix}

\appendix
\renewcommand{\appendixname}{Appendix~\Alph{section}}

\section{Proof of Proposition \ref{finite dim convergence}} \label{detailed proof}

\begin{proof} 
	To prove the case when $n=2$, by the Markov property, we first have:
	$$
	\prob^0(X_{t_1}^\delta \in (a_1, b_1], X^\delta_{t_2}\in (a_2,b_2])=\int_{a_1}^{b_1} \prob^y( X^\delta_{t_2-t_1}\in (a_2,b_2])f^\delta(t_1,y)dy.
	$$
	and 
	$$
	\prob^0(X_{t_1} \in (a_1, b_1], X_{t_2}\in (a_2,b_2])=\int_{a_1}^{b_1} \prob^y( X_{t_2-t_1}\in (a_2,b_2])f(t_1,y)dy.
	$$
	Thus 
	\bae
	&\prob^0(X_{t_1}^\delta \in (a_1, b_1], X^\delta_{t_2}\in (a_2,b_2]) - \prob^0(X_{t_1} \in (a_1, b_1], X_{t_2}\in (a_2,b_2]) \\
	\le& \left| \int_{a_1}^{b_1} \prob^y( X_{t_2-t_1}\in (a_2,b_2]) \left[f^\delta(t_1,y)-f(t_1,y)\right] dy \right| \\
	+& \int_{a_1}^{b_1} \left| \prob^y( X^\delta_{t_2-t_1}\in (a_2,b_2]) - \prob^y( X_{t_2-t_1}\in (a_2,b_2]) \right| f^\delta(t_1,y) dy \\
	=&: I_1 +I_2
	\eae
	For $I_1$, consider the function 
	\beq
	\label{test function}
	\varphi(y)=
	\begin{cases}
		\prob^y(X_{t_2-t_1} \in (a_2,b_2])& \text{if y $\in (a_1,b_1]$},\\
		0& \text{otherwise}.
	\end{cases}
	\eeq
	According to the strong Feller property of $X_t$, $\prob^y(X_{t_2-t_1} \in (a_2,b_2])$ is a continuous function with respect to $y \in (-\infty,1)$. Thus $\varphi(y)$ is a bounded measurable function with discontinuities at $a_1$ and $b_1$. Moreover, by Proposition 3.1 of \cite{liu2020rigorous}, $F(t,y)$ is continuous in $y$ and thus puts $0$ measure on $\{a_1, b_1\}$. Thus by Theorem 3.2.10 of \cite{durrett2019probability} and the fact that $\e^0 [\varphi(X_{t_1}^\delta)] \to \e^0 [\varphi(X_{t_1})]$ as $\delta \to 0^+$, we get $I_1 \to 0$.
	
	For $I_2$, by Corollary \ref{uniform convergence in distribution}, for all $\delta\le \eta_0$ we have 
	\beq
	\label{second term}
	I_2\le \epsilon \int_{a_1}^{b_1}f^\delta(t,y)dy\le \epsilon.
	\eeq
	Thus we have proved Proposition 1 for $n=2$. In general, suppose Proposition \ref{finite dim convergence} holds for all $k\le n$. Now for $k=n+1$, we can similarly define $\varphi^{(n)}(y)$ for $y=(y_1, \cdots, y_n)$ as 
	\beq
	\varphi^{(n)}(y)=
	\begin{cases}
		\prob^y(X_{t_{n+1}-t_n} \in (a_{n+1},b_{n+1}])& \text{if $y \in (a_1,b_1]\times \cdots \times (a_n, b_n]$},\\
		0& \text{otherwise}.
	\end{cases}
	\eeq
	and have 
	\begin{equation}
	\begin{aligned}
	&\prob^0(X^\delta_{t_1}\in (a_1,b_1], X^\delta_{t_2}\in (a_2,b_2], \cdots, X^\delta_{t_n}\in (a_n,b_n], X^\delta_{t_{n+1}}\in (a_{n+1},b_{n+1}])\\
	-&\prob^0(X_{t_1}\in (a_1,b_1], X_{t_2}\in (a_2,b_2], \cdots, X_{t_n}\in (a_n,b_n], X_{t_{n+1}}\in (a_{n+1},b_{n+1}])\\
	\le& \left|\e^0\left[ \varphi^{(n)}(X_{t_1}, \cdots, X_{t_1}) \right] -\e^0\left[ \varphi^{(n)}(X^\delta_{t_1}, \cdots, X^\delta_{t_1}) \right]  \right|\\
	+& \int_{(a_1,b_1]\times \cdots \times (a_n, b_n]} \left| \prob^{y}( X^\delta_{t_{n+1}-t_{n}}\in (a_{n+1},b_{n+1}]) - \prob^{y}( X_{t_{n+1}-t_n}\in (a_{n+1},b_{n+1}]) \right| dF^{(n)}(y).
	\end{aligned}
	\end{equation}
	where $F^{(n)}(y)$ is the distribution of $(X_{t_1}, \cdots, X_{t_n})$. Note that $\varphi^{(n)}(\cdot)$ is a bounded measurable function on $\mathbb{R}^n$ whose discontinuities are given by 
	$$
	D_{\varphi^{(n)}} = \{y: \{y_1=a_1 \text{ or } b_1\}\cup \cdots \cup \{y_n=a_n \text{ or } b_n\} \}
	$$
	Since $F(\cdot, t_i)$, $i=1,2,\cdots n$ are all continuous on $(-\infty, 1]$, the joint distribution $F^{(n)}(\cdot)$ of $(X_{t_1}, \cdots, X_{t_n})$ puts $0$ mass on $D_{\varphi^{(n)}}$. Thus by (\romannumeral6) of Theorem 3.10.1 in \cite{durrett2019probability}, we have 
	$$
	\left|\e^0\left[ \varphi^{(n)}(X_{t_1}, \cdots, X_{t_1}) \right] -\e^0\left[ \varphi^{(n)}(X^\delta_{t_1}, \cdots, X^\delta_{t_1}) \right]  \right| \to 0.
	$$ 
	At the same time, by a similar argument as in \eqref{second term}, we have 
	$$
	\int_{(a_1,b_1]\times \cdots \times (a_n, b_n]} \left| \prob^{y}( X^\delta_{t_{n+1}-t_{n}}\in (a_{n+1},b_{n+1}]) - \prob^{y}( X_{t_{n+1}-t_n}\in (a_{n+1},b_{n+1}]) \right| dF^{(n)}(y) \le \epsilon.
	$$
	Thus by induction, the proof of Proposition \ref{finite dim convergence} is complete.
\end{proof}

\section{Proof of Lemma \ref{event control}}\label{detail}

\begin{proof}
	Recall Lemma \ref{OU control}, the definition of $E^{\epsilon_n}$ and the OU process
	$$
	Z^{1+\epsilon_n}_t=1+\epsilon_n - \int_0^tZ_s^{1+\epsilon_n}ds +B_t.
	$$
	In event $E^{\epsilon_n}$ there is
	\beq
	\label{n-1 stopping time upper bound}
	Z_{\Gamma_{4\epsilon_n}}^{1+\epsilon_n} \ge 1 + \epsilon_n - 2 \cdot 4\epsilon_n^2 + 4\epsilon_n >1 + 2\epsilon_n = 1+ \epsilon_{n-1}.
	\eeq
	Thus we know that $\tau_{n-1} < \Gamma_{4\epsilon_n}$ in event $E^{\epsilon_n}$. And for any $t<\Gamma_{4\epsilon_n}<4\epsilon_n^2$ with $\Gamma_{4\epsilon_n} < \Gamma_{-\frac12\epsilon_n}$,
	$$
	Z_t^{1+\epsilon_n} \ge 1 + \epsilon_n - 2\cdot 4\epsilon_n^2- \frac12\epsilon_n > 1 - 2\epsilon_n = 1 - \epsilon_{n-1}.	
	$$
	Thus we know that $Z_t^{1+\epsilon_n}$ does not hit $1-\epsilon_{n-1}$ before $\Gamma_{ 4\epsilon_n }$. Thus it must hit $1+\epsilon_{n-1}$ before $\Gamma_{ 4\epsilon_n }$. However, note that $\forall$ $t<\Gamma_{ \frac23\epsilon_n } < \Gamma_{ 4\epsilon_n }$, 
	$$
	Z_t^{1+\epsilon_n} \le 1 + \epsilon_n -0 - \frac23\epsilon_n  < 1 + 2\epsilon_n = 1+ \epsilon_{n-1}.	
	$$
	Thus we have $\Gamma_{\frac23\epsilon_n} < \tau_{n-1} < \Gamma_{4\epsilon_n}$ in $E^{\epsilon_n}$. And in $E^{\epsilon_n}$ for any $t<\Gamma_{4\epsilon_n}$ such that $B_t \in [\frac14\epsilon_n, \frac12 \epsilon_n]$,
	$$
	Z_t^{1+\epsilon_n} \ge 1 + \epsilon_n - 2\cdot 4\epsilon_n^2 + \frac14\epsilon_n  > 1 + \epsilon_n.	
	$$
	Thus in event $E^{\epsilon_n}$
	$$
	\int_0^{\tau_{n-1}} \mathbbm{1}_{ \{Z_t^{1+\epsilon_n} > 1+\epsilon_n\} }dt \ge \int_0^{ \Gamma_{\frac23\epsilon_n} } \mathbbm{1}_{ \{B_t\in [\frac14 \epsilon_n, \frac12\epsilon_n]\} } dt \ge \epsilon_n^2,
	$$
	and then
	$$
	\prob \left( \int_0^{\tau_{n-1}} \mathbbm{1}_{ \{Z_t^{1+\epsilon_n}>1+\epsilon_n\} }dt>\epsilon_n^2 \right) \ge \prob(E^{\epsilon_n} ). 
	$$
	Similarly, given $E^{\epsilon_n}$, for all $t < \Gamma_{4\epsilon_n} < 4\epsilon_n^2$, since $\Gamma_{4 \epsilon_n} < \Gamma_{-\frac12 \epsilon_n}$, we have 
	$$
	Z_t^{1-\epsilon_n} \ge 1-\epsilon_n - 2 \cdot 4\epsilon_n^2 -\frac12 \epsilon_n > 1-2\epsilon_n = 1-\epsilon_{n-1}
	$$
	and 
	$$
	Z_{\Gamma_{4\epsilon_n}}^{1 - \epsilon_n} \ge 1-\epsilon_n -2\cdot 4\epsilon_n^2 + 4\epsilon_n  > 1+2\epsilon_n = 1+\epsilon_{n-1}.
	$$
	At the same time, $\forall t< \Gamma_{2.8\epsilon_n} < \Gamma_{4\epsilon_n}$, we have 
	$$
	Z_t^{1-\epsilon_n} \le 1-\epsilon_n - 0 + 2.8 \epsilon_n  < 1+2\epsilon_n = 1+ \epsilon_{n-1}.
	$$
	Thus $\Gamma_{2.8\epsilon_n} < \tau_{n-1} < \Gamma_{4\epsilon_n}$. And for any $ t < \Gamma_{4\epsilon_n}$, $B_t \in [2.2 \epsilon_n, 2.5\epsilon_n]$, we also have 
	$$
	Z^{1-\epsilon_n}_t \ge 1-\epsilon_n  - 2 \cdot 4\epsilon_n^2+ 2.2\epsilon_n > 1 + \epsilon_n.
	$$
	So again in $ E^{\epsilon_n}$ we have  
	$$
	\int_0^{\tau_{n-1}} \mathbbm{1}_{ \{Z_t^{1-\epsilon_n}> 1 + \epsilon_n\} }dt \ge \int_0^{\Gamma_{2.8\epsilon_n}} \mathbbm{1}_{ \{B_t\in [2.2 \epsilon_n, 2.5 \epsilon_n]\} } dt \ge \epsilon_n^2,
	$$
	and then
	$$
	\prob \left( \int_0^{\tau_{n-1}} \mathbbm{1}_{ \{Z_t^{1-\epsilon_n}>1 + \epsilon_n\} }dt>\epsilon_n^2 \right) \ge \prob(E^{\epsilon_n} ). 
	$$
\end{proof}




\begin{thebibliography}{10}



\bibitem{MR2353702}
Iddo Ben-Ari and Ross~G. Pinsky.
\newblock Spectral analysis of a family of second-order elliptic operators with
  nonlocal boundary condition indexed by a probability measure.
\newblock {\em J. Funct. Anal.}, 251(1):122--140, 2007.
\url{https://doi.org/10.1016/j.jfa.2007.05.019}

\bibitem{MR2499861}
Iddo Ben-Ari and Ross~G. Pinsky.
\newblock Ergodic behavior of diffusions with random jumps from the boundary.
\newblock {\em Stochastic Process. Appl.}, 119(3):864--881, 2009.
\url{https://doi.org/10.1016/j.spa.2008.05.002}

\bibitem{brunel2000dynamics}
Nicolas Brunel.
\newblock Dynamics of sparsely connected networks of excitatory and inhibitory
  spiking neurons.
\newblock {\em J. Comput. Neurosci.}, 8(3):183--208, 2000.
\url{https://doi.org/10.1023/A:1008925309027}

\bibitem{brunel1999fast}
Nicolas Brunel and Vincent Hakim.
\newblock Fast global oscillations in networks of integrate-and-fire neurons
  with low firing rates.
\newblock {\em Neural Comput.}, 11(7):1621--1671, 1999.
\url{https://doi.org/10.1162/089976699300016179}

\bibitem{caceres2011analysis}
Mar{\'\i}a~J C{\'a}ceres, Jos{\'e}~A Carrillo, and Beno{\^\i}t Perthame.
\newblock Analysis of nonlinear noisy integrate \& fire neuron models: blow-up
  and steady states.
\newblock {\em  J. Math. Neurosci.}, 1(1):7, 2011.
\url{https://doi.org/10.1186/2190-8567-1-7}

\bibitem{caceres2011numerical}
Mar{\'\i}a~J C{\'a}ceres, Jos{\'e}~A Carrillo, and Louis Tao.
\newblock A numerical solver for a nonlinear fokker--planck equation
  representation of neuronal network dynamics.
\newblock {\em J. Comput. Phys.}, 230(4):1084--1099, 2011.
\url{https://doi.org/10.1016/j.jcp.2010.10.027}

\bibitem{Caceres2014}
Mar\'{\i}a~J. C\'{a}ceres and Beno\^{\i}t Perthame.
\newblock Beyond blow-up in excitatory integrate and fire neuronal networks:
  refractory period and spontaneous activity.
\newblock {\em J. Theoret. Biol.}, 350:81--89, 2014.
\url{https://doi.org/10.1016/j.jtbi.2014.02.005}

\bibitem{Caceres2018}
Mar\'{\i}a~J. C\'{a}ceres and Ricarda Schneider.
\newblock Analysis and numerical solver for excitatory-inhibitory networks with
  delay and refractory periods.
\newblock {\em ESAIM Math. Model. Numer. Anal.}, 52(5):1733--1761, 2018.
\url{https://doi.org/10.1051/m2an/2018014}

\bibitem{carrillo2013classical}
Jos{\'e}~A Carrillo, Mar{\'\i}a D.~M. Gonz{\'a}lez, Maria~P Gualdani, and Maria~E
  Schonbek.
\newblock Classical solutions for a nonlinear Fokker-Planck equation arising in
  computational neuroscience.
\newblock {\em Comm. Partial Differential Equations},
  38(3):385--409, 2013.
  \url{https://doi.org/10.1080/03605302.2012.747536}

\bibitem{Cercignani2001Rarefied}
Carlo Cercignani.
\newblock {\em Rarefied gas dynamics: From basic concepts to actual calculations.}
\newblock Cambridge University Press, 2000.

\bibitem{compte2000synaptic}
Albert Compte, Nicolas Brunel, Patricia~S. Goldman-Rakic, and Xiao-Jing Wang.
\newblock Synaptic Mechanisms and Network Dynamics Underlying Spatial Working
  Memory in a Cortical Network Model.
\newblock {\em Cereb. Cortex}, 10(9):910--923, 2000.
\url{https://doi.org/10.1093/cercor/10.9.910}


\bibitem{cormier2020hopf}
Quentin Cormier, Etienne Tanr{\'e}, and Romain Veltz.
\newblock Hopf bifurcation in a Mean-Field model of spiking neurons.
\newblock {\em Electron. J. Probab.}, 26(art. 121):1--40, 2021.
\url{https://doi.org/10.1214/21-ejp688}

\bibitem{cormier2020long}
Quentin Cormier, Etienne Tanr{\'e}, and Romain Veltz.
\newblock Long time behavior of a mean-field model of interacting neurons.
\newblock {\em Stochastic Process. Appl.}, 130(5):2553--2595,
  2020.
  \url{https://doi.org/10.1016/j.spa.2019.07.010}

\bibitem{Dautray1988Mathematical}
Robert Dautray and Jacques~Louis Lions.
\newblock {\em Mathematical analysis and numerical methods for science and
  technology. Vol. 2.  Functional and variational methods.}
\newblock Springer-Verlag, Berlin, 1988.
\url{https://doi.org/10.1007/978-3-642-61566-5}

\bibitem{Delarue2015}
Fran{\c{c}}ois Delarue, James Inglis, Sylvain Rubenthaler, and Etienne
Tanr{\'e}.
\newblock Particle systems with a singular mean-field self-excitation.
  {A}pplication to neuronal networks.
\newblock {\em Stochastic Process. Appl.}, 125(6):2451--2492, 2015.
\url{https://doi.org/10.1016/j.spa.2015.01.007}

\bibitem{delarue2013first}
Fran{\c{c}}ois Delarue, James Inglis, Sylvain Rubenthaler, and Etienne
  Tanr{\'e}.
\newblock First hitting times for general non-homogeneous 1d diffusion
  processes: Density estimates in small time.
  \newblock {\em Tech. Report.}, 2013.
\url{http://hal.archives-ouvertes.fr/hal-00870991}

\bibitem{Desvillettes2011The}
Laurent Desvillettes and Mario Pulvirenti.
\newblock The linear Boltzmann equation for long-range forces: a derivation
  from particle systems.
\newblock {\em Math. Models Methods Appl. Sci.}, 9(8):1123--1145,
  1999.
  \url{https://doi.org/10.1142/S0218202599000506}

\bibitem{durrett2018stochastic}
Richard Durrett.
\newblock {\em Stochastic calculus: a practical introduction}.  Probab. Stochastics Ser. 
\newblock CRC press, 1996.

\bibitem{durrett2019probability}
Rick Durrett.
\newblock {\em Probability: theory and examples}, volume~49.
\newblock Cambridge university press, 2019.
\url{https://doi.org/10.1017/9781108591034}

\bibitem{Erdos2015Linear}
L\'{a}szl\'{o} Erdos and Horng-Tzer Yau.
\newblock Linear {B}oltzmann equation as the weak coupling limit of a random 
{S}chr\"{o}dinger equation.
\newblock {\em Comm. Pure Appl. Math.}, 53(6):667--735,
  2000.
  \url{https://doi.org/10.1002/(SICI)1097-0312(200006)53:6<667::AID-CPA1>3.0.CO;2-5}

\bibitem{ethier2009markov}
Stewart N. Ethier and Thomas G. Kurtz.
\newblock {\em Markov processes: Characterization and convergence}, volume 282.
\newblock John Wiley \& Sons, 1986.
\url{https://doi.org/10.1002/9780470316658}

\bibitem{MR47886}
William Feller.
\newblock The parabolic differential equations and the associated semi-groups
  of transformations.
\newblock {\em Ann. of Math. (2)}, 55:468--519, 1952.
\url{https://doi.org/10.2307/1969644}

\bibitem{MR63607}
William Feller.
\newblock Diffusion processes in one dimension.
\newblock {\em Trans. Amer. Math. Soc.}, 77:1--31, 1954.
\url{https://doi.org/10.2307/1990677}

\bibitem{friedman2008partial}
Avner Friedman.
\newblock {\em Partial differential equations of parabolic type}.
\newblock Englewood Cliffs, NJ: Prentice-Hall, Inc., 1964.

\bibitem{garroni1992green}
Maria~Giovanna Garroni and Jos{\'e}~Luis Menaldi.
\newblock {\em Green functions for second order parabolic integro-differential
  problems}, Pitman Research Notes in Mathematics Series \textbf{275}. Harlow: Longman 
  Scientific \& Technical; copublished in the United States with Wiley, New York, 1992. 

\bibitem{gerstner2002spiking}
Wulfram Gerstner and Werner~M Kistler.
\newblock {\em Spiking neuron models: Single neurons, populations, plasticity}.
\newblock Cambridge university press, 2002.
\url{https://doi.org/10.1017/CBO9780511815706}

\bibitem{MR0346904}
\u{I}osip \={I}llich G\={\i}hman and Anatolii Volodimirovich Skorohod.
\newblock {\em Stochastic differential equations}.
\newblock Springer-Verlag, New York-Heidelberg, 1972.

\bibitem{guillamon2004introduction}
Toni Guillamon.
\newblock An introduction to the mathematics of neural activity.
\newblock {\em Butl. Soc. Catalana Mat}, 19(2):25--45, 2004.

\bibitem{liu2020rigorous}
Jian Guo Liu, Ziheng Wang, Yuan Zhang, and Zhennan Zhou.
\newblock Rigorous justification of the Fokker-Planck equations of neural
  networks based on an iteration perspective, \ARXIV{2005.08285}, 2021.

\bibitem{hanson2007applied}
Floyd~B. Hanson.
\newblock {\em Applied stochastic processes and control for jump-diffusions:
  Modeling, analysis and computation}.
\newblock SIAM, 2007.
\url{https://doi.org/10.1137/1.9780898718638}

\bibitem{Yantong}
Jingwei Hu, Jian-Guo Liu, Yantong Xie, and Zhennan Zhou.
\newblock A structure preserving numerical scheme for Fokker-Planck equations of neuron 
networks: Numerical analysis and exploration
\newblock {\em J. Comput. Phys.}, 433 (pap.23), 2021
\url{https://doi.org/10.1016/j.jcp.2021.110195}

\bibitem{ikeda2021theoretical}
Kota Ikeda, Pierre Roux, Delphine Salort, and Didier Smets.
\newblock Theoretical study of the emergence of periodic solutions for the
  inhibitory NNLIF neuron model with synaptic delay.
\newblock 2021.
\url{https://hal.archives-ouvertes.fr/hal-03157218/}

\bibitem{Inglis2015}
James Inglis and Denis Talay.
\newblock Mean-field limit of a stochastic particle system smoothly interacting
  through threshold hitting-times and applications to neural networks with
  dendritic component.
\newblock {\em SIAM J. Math. Anal.}, 47(5):3884--3916, 2015.
\url{https://doi.org/10.1137/140989042}

\bibitem{jahn2011motoneuron}
Patrick Jahn, Rune~W. Berg, J{\o}rn Hounsgaard, and Susanne Ditlevsen.
\newblock Motoneuron membrane potentials follow a time inhomogeneous jump
  diffusion process.
\newblock {\em J. Comput. Neurosci.}, 31(3):563--579, 2011.
\url{https://doi.org/10.1007/s10827-011-0326-z}

\bibitem{lapique1907recherches}
Louis Lapicque.
\newblock Recherches quantitatives sur l’excitation
électrique des nerfs traitée comme une polarisation
\newblock {\em Journal de
Physiologie et de Pathologie Générale} 9: 620–635, 1907.

\bibitem{Liggett2010Continuous}
Thomas~M. Liggett.
\newblock {\em Continuous Time Markov Processes: An Introduction}.
\newblock Graduate Studies in Mathematics, vol. 113., American Mathematical Society,
\newblock 2010.
\url{https://doi.org/10.1090/gsm/113}

\bibitem{mattia2002population}
Maurizio Mattia and Paolo Del Giudice.
\newblock Population dynamics of interacting spiking neurons.
\newblock {\em Phys. Rev. E (3)}, 66(5):051917, 2002.
\url{https://doi.org/10.1103/PhysRevE.66.051917}

\bibitem{newhall2010cascade}
Katherine A. Newhall, Gregor Kova\v{c}i\v{c}, Peter R. Kramer, and David Cai.
\newblock Cascade-induced synchrony in stochastically driven neuronal networks.
\newblock {\em Phys. Rev. E (3)}, 82(4):041903, 2010.
\url{https://doi.org/10.1103/PhysRevE.82.041903}

\bibitem{newhall2010dynamics}
Katherine A. Newhall, Gregor Kova\v{c}i\v{c}, Peter R. Kramer, Douglas Zhou, Aaditya V.
  Rangan, and David Cai.
\newblock Dynamics of current-based, Poisson driven, integrate-and-fire
  neuronal networks.
\newblock {\em Commun. Math. Sci.}, 8(2):541--600, 2010.
\url{http://projecteuclid.org/euclid.cms/1274816894}

\bibitem{Nykamp_thesis}
Duane~Quinn Nykamp.
\newblock {\em A population density approach that facilitates large-scale
  modeling of neural networks}.
\newblock Thesis (Ph.D.), Courant Institute of Mathematical Sciences, New York University, 
2000.

\bibitem{Oksendal1998Stochastic}
Bernt Oksendal.
\newblock {\em Stochastic differential equations: An introduction with applications}.
\newblock Universitext,  Springer-Verlag, Berlin, 2998.
\url{https://doi.org/10.1007/978-3-662-03620-4}

\bibitem{omurtag2000simulation}
Ahmet Omurtag, Bruce~W. Knight, and Lawrence Sirovich.
\newblock On the simulation of large populations of neurons.
\newblock {\em J. Comput. Neurosci.}, 8(1):51--63, 2000.
\url{https://doi.org/10.1023/A:1008964915724}

\bibitem{MR3244555}
Jun Peng.
\newblock A note on the first passage time of diffusions with holding and
  jumping boundary.
\newblock {\em Statist. Probab. Lett.}, 93:58--64, 2014.
\url{https://doi.org/10.1016/j.spl.2014.06.012}

\bibitem{MR3016590}
Jun Peng and WenBo~V. Li.
\newblock Diffusions with holding and jumping boundary.
\newblock {\em Sci. China Math.}, 56(1):161--176, 2013.
\url{https://doi.org/10.1007/s11425-012-4416-9}

\bibitem{MR573203}
Philip Protter.
\newblock Stochastic differential equations with jump reflection at the
  boundary.
\newblock {\em Stochastics}, 3(3):193--201, 1980.
\url{https://doi.org/10.1080/17442508008833144}

\bibitem{renart2004mean}
Alfonso Renart, Nicolas Brunel, and Xiao-Jing Wang.
\newblock Mean-field theory of irregularly spiking neuronal populations and
  working memory in recurrent cortical networks.
\newblock {\em in Computational neuroscience: A comprehensive approach}, pp.431--490,
ed. by J. Feng, CRC Press, 2004.
  \url{https://doi.org/10.1201/9780203494462}.

\bibitem{sirovich2006dynamics}
Lawrence Sirovich, Ahmet Omurtag, and Kip Lubliner.
\newblock Dynamics of neural populations: Stability and synchrony.
\newblock {\em Network}, 17(1):3--29, 2006.
\url{https://doi.org/10.1080/09548980500421154}


\bibitem{MR2673971}
Leszek S\l~omi\'{n}ski and Tomasz Wojciechowski.
\newblock Stochastic differential equations with jump reflection at
  time-dependent barriers.
\newblock {\em Stochastic Process. Appl.}, 120(9):1701--1721, 2010.
\url{https://doi.org/10.1016/j.spa.2010.04.008}

\bibitem{Touboul2012}
Jonathan Touboul, Geoffroy Hermann, and Olivier Faugeras.
\newblock Noise-induced behaviors in neural mean field dynamics.
\newblock {\em SIAM J. Appl. Dyn. Syst.}, 11(1):49--81, 2012.
\url{https://doi.org/10.1137/110832392}

\bibitem{tuckwell1988introduction}
Henry~C Tuckwell.
\newblock {\em Introduction to theoretical neurobiology: volume 2, Nonlinear
  and Stochastic Theories}, 
Cambridge Studies in Mathematical Biology, vol. 8.
\newblock Cambridge University Press, 1988.
\url{https://doi.org/10.1017/CBO9780511623202}

\end{thebibliography}


\providecommand{\bysame}{\leavevmode\hbox to3em{\hrulefill}\thinspace}
\providecommand{\MR}{\relax\ifhmode\unskip\space\fi MR }
\providecommand{\MRhref}[2]{%
  \href{http://www.ams.org/mathscinet-getitem?mr=#1}{#2}
}
\providecommand{\href}[2]{#2}


\ACKNO{This work has been partially supported by Beijing Academy of Artificial Intelligence (BAAI). Z. Zhou is supported by NSFC grant No. 11801016, No. 12031013 and the National Key $R \& D$ Program of China, project Number 2020YFA0712000. J-G. Liu is partially supported by NSF grants DMS 1812573 and DMS 2106988. Y.Zhang is supported by NSFC 12026606 and the National Key $R \& D$ Program of China, project Number 2020YFA0712902. Z. Zhou thanks Beno\^{\i}t Perthame for helpful discussions. The authors would like to extend their profound gratitude to Dr. Lihu Xu for teaching them the nice and easy proof of the strong Feller property in Section \ref{feller section}.}

\nocite{*}
\end{document}